\documentclass[reqno]{amsart}

\usepackage[hang,flushmargin]{footmisc}

\usepackage{xcolor}

\usepackage{amssymb}
\usepackage{amsthm}
\usepackage{amsmath}
\numberwithin{equation}{section}
\usepackage{mathtools}

\mathtoolsset{showonlyrefs}

\usepackage{amsfonts}

\usepackage{hyperref}
\usepackage{nameref}

\usepackage{graphicx,rotating}

\usepackage[font=footnotesize,labelfont=footnotesize]{caption}
\usepackage[utf8]{inputenc}

\usepackage{enumerate}

\usepackage{pdfsync}
\usepackage{caption}

\theoremstyle{plain}
\newtheorem{theorem}{Theorem}[subsection]
\newtheorem*{thm}{Theorem}
\newtheorem{corollary}[theorem]{Corollary}
\newtheorem{lemma}[theorem]{Lemma}
\newtheorem{proposition}[theorem]{Proposition}

\newtheorem{prop}{Proposition}[section]
\newtheorem{lem}{Lemma}[section]

\theoremstyle{remark}
\newtheorem*{remark}{Remark}

\newcommand{\R}{\mathbb{R}}
\newcommand{\C}{\mathbb{C}}
\newcommand{\Rd}{\mathbb{R}^d}
\newcommand{\Z}{\mathbb{Z}}

\newcommand{\Lt}[1][d]{L^2(\R^{#1})}

\renewcommand{\l}{\lambda}
\renewcommand{\L}{\Lambda}

\newcommand{\vol}{\textnormal{vol}}

\DeclareMathOperator*{\supp}{supp}
\DeclareMathOperator*{\sinc}{sinc}

\begin{document}

\title[]{Maximal polarization for periodic configurations on the real line}

\author[]{Markus Faulhuber }
\address{NuHAG, Faculty of Mathematics, University of Vienna, Oskar-Morgenstern-Platz 1, 1090 Vienna, Austria}
\email{markus.faulhuber@univie.ac.at}

\author[]{Stefan Steinerberger }
\address{Department of Mathematics, University of Washington, Seattle, WA 98195-4350, USA}
\email{steinerb@uw.edu}

\begin{abstract}
    We prove that among all 1-periodic configurations $\Gamma$ of points on the real line $\mathbb{R}$ the quantities
    $$ \min_{x \in \mathbb{R}} \sum_{\gamma \in \Gamma} e^{- \pi \alpha (x - \gamma)^2} \quad \text{and} \quad  \max_{x \in \mathbb{R}} \sum_{\gamma \in \Gamma} e^{- \pi \alpha (x - \gamma)^2}$$
    are maximized and minimized, respectively, if and only if the points are equispaced and whenever the number of points $n$ per period is sufficiently large (depending on $\alpha$). This solves the polarization problem for periodic configurations with a Gaussian weight on $\mathbb{R}$ for large $n$. The first result is shown using Fourier series. The second result follows from work of Cohn and Kumar on universal optimality and holds for all $n$ (independent of $\alpha$).
\end{abstract}

\subjclass[2020]{52C25, 74G65, 82B21} 
\keywords{equispaced points, Gaussian, periodic configuration, polarization}
\thanks{M.F.\ is supported by the Austrian Science Fund (FWF) grant P33217. S.S.\ is supported by the NSF (DMS-2123224) and the Alfred P.\ Sloan Foundation.}

\maketitle

\section{Introduction and main result}\label{sec:intro}
We study the following question: for fixed $\alpha > 0$, among all periodic configurations of points $\Gamma$ with given density on the real line, for which one is the function
\begin{equation}\label{eq:p_alpha}
    p_\alpha(x) = \sum_{\gamma \in \Gamma} e^{-\pi \alpha (x - \gamma)^2}
\end{equation}
as close to constant as possible?  Factoring out scales, periodicity and symmetries, this is equivalent to the problem of placing $n$ points on $\mathbb{T} \cong \mathbb{S}^1$ so that
\begin{equation}\label{eq:f_alpha}
    f_\alpha(x) =\sum_{j=1}^n  \sum_{k \in \Z} e^{-\pi \alpha k^2} e^{2 \pi i k (x - x_j)}
\end{equation}
is as close to constant as possible. The equivalence of the two problems arises from the duality between \eqref{eq:p_alpha} and \eqref{eq:f_alpha} caused by the Poisson Summation Formula, which we explain in detail in \S \ref{sec:notation}. We note that $\eqref{eq:f_alpha}$ can be expressed by means of the Jacobi theta function $\theta(x;\alpha)$ (details are given in \S \ref{sec:notation});
\begin{equation}
    \theta_\alpha (x) = \theta(x;\alpha) = \sum_{k \in \Z} e^{-\pi \alpha k^2} e^{2 \pi i k x}.
\end{equation}
The problem arises naturally in a variety of settings, see \S \ref{sec:results}. Such problems are often related to optimal sphere packing/covering.
Since sphere packing in one dimension is trivial, one would expect equispaced points to be optimal. Indeed, Cohn and Kumar \cite{CohKum07} showed that equispaced points on the line are \textit{universally optimal}. Their result can be applied in our setting.

\begin{prop}[Application of Cohn and Kumar \cite{CohKum07}]\label{pro:Cohn-Kumar}
    Among all periodic configurations $\Gamma \subset \R$ of the form
    $$\Gamma = \frac{n}{\delta} \left(\bigcup_{k=1}^n(\Z+x_k) \right), \qquad \{x_1, \ldots , x_n \} \subset [0,1), \, x_k \neq x_\ell, \, k \neq \ell,$$
    of fixed density $\delta > 0$  and for any fixed parameter $\alpha > 0$, the quantity
    $$\max_{x \in \R} \, p_\alpha(x) = \max_{x \in \R} \sum_{\gamma \in \Gamma} e^{-\pi \alpha (x-\gamma)^2} \qquad \text{ is minimized}$$
    if and only if the points are equispaced. Moreover, among all sets of $n$ points on the torus $\mathbb{T} \cong \mathbb{S}^1$ and for any fixed parameter $\alpha > 0$, the quantity
    $$\max_{x \in \mathbb{T}} f_\alpha(x) = \max_{x \in \mathbb{T}}\sum_{j=1}^n  \theta_{\alpha}(x - x_j) \qquad \mbox{is minimized}$$
    if and only if the points are equispaced.
\end{prop}
This result is not surprising, it is exactly what one would expect. However, to the best of our knowledge no ``easy" proof of the theorem of Cohn and Kumar is known. As a consequence, since our proof of Proposition \ref{pro:Cohn-Kumar} makes use of the result of Cohn and Kumar, we do not currently have an ``elementary'' proof. We refer to \S \ref{sec:energy} for an in-depth discussion of this result and give the proof in \S \ref{sec:proof_Cohn-Kumar}.

\medskip

Proposition \ref{pro:Cohn-Kumar} is concerned with minimizing the maximum.
The main result of our paper is the dual, maximizing the minimum, which we prove in the regime when the number of points is sufficiently large, where ``sufficiently large" depends only on the width $\alpha$ of the Gaussian. The proof is given in \S \ref{sec:proof_main}.

\begin{thm}[Main Result]
    For $n \geq N(\alpha)$ (depending only on $\alpha$) and among all 1-periodic configurations $\Gamma \subset \R$ of density $n$, i.e.,
    \begin{equation}
        \Gamma = \bigcup_{k=1}^n (\Z + x_k), \qquad \{x_1, \ldots , x_n \} \subset [0,1), \, x_k \neq x_\ell, \, k \neq \ell,
    \end{equation}
    the quantity
    \begin{equation}
        \min_{x \in \R} p_\alpha(x) = \min_{x \in \R} \sum_{\gamma \in \Gamma} e^{-\pi \alpha (x-\gamma)^2}, \, \alpha > 0, \qquad \text{is maximized}
    \end{equation}
    if and only if the points are equispaced. Moreover, for $n \geq N(\alpha)$ sufficiently large (depending only on $\alpha$)
    $$ \min_{x \in \mathbb{T}} f_\alpha(x) = \min_{x \in \mathbb{T}}\sum_{j=1}^n  \theta_{\alpha}(x - x_j) \qquad \text{is maximized}$$
    if and only if the points are equispaced.
\end{thm}
Just as in Proposition \ref{pro:Cohn-Kumar}, the two statements are dual by the Poisson Summation Formula. We remark that the parameter $\alpha$ in the result for $p_\alpha$ corresponds to $1/\alpha$ in the statement for $f_\alpha$ (see \S \ref{sec:notation}). Note that the results are invariant under global shifts $z$ as the sets
$\{x_1, \ldots , x_n\}$ and $\{x_1+z, \ldots , x_n +z \}$
both yield the same \textit{energy} and \textit{polarization} (see \S \ref{sec:results}). Also, \textit{equispaced} is always understood periodically. The argument is structurally completely different from the Cohn and Kumar framework \cite{CohKum07} of universal optimality, the proof invokes very different tools.
The main obstacles when establishing our results are:
\begin{enumerate}[(a)]
    \item the location of the minimum depends on the $x_j$ in a complicated way and
    \item for equispaced points the difference between minimum and mean is super-exponentially small in $n$, which forces an analysis on very small scales.
\end{enumerate}

The proof of the main result is completely Fourier-analytic which makes it somewhat robust and applicable to a wider range of functions than just the Gaussian function; if one has, generally, a function of the type
	$$g(x) = \sum_{k \in \Z} \widehat{g}(|k|) e^{2 \pi i k x},$$
with $\widehat{g}(|k|)$ decaying sufficiently fast (say, faster than exponential), then much (but not all) of the argument carries over verbatim. For simplicity of exposition, the remainder of the paper only deals with the Gaussian case which is arguably the most natural. The proof is explicit enough that bounds on $N(\alpha)$ could be obtained, however, since one would naturally assume that the result is true for all $n \geq 1$, independently of the value of $\alpha$, we will not track this dependency. The condition $n \geq N(\alpha)$ is necessary in many different steps of our argument and it appears that an unconditional argument for all $n \geq 1$ would require some new ideas. Of course the case $n=1$ is trivial and we provide a proof valid for all $\alpha > 0$ when $n=2$ in \S \ref{sec:n}. It appears that already the case $n=3$ poses some nontrivial difficulties.

\section{Related results}\label{sec:results}

\subsection{Energy minimization}\label{sec:energy}
Energy minimization problems have received much attention in recent years. A seminal result due to Cohn, Kumar, Miller, Radchenko, and Viazovska \cite{Coh-Via22} states that the $\mathsf{E}_8$-lattice and Leech lattice are universally optimal in their respective dimension, meaning that they uniquely minimize energy $E_g(\Gamma)$ among periodic configurations $\Gamma$ and for a large class of (radial) potential functions $g$. A periodic configuration in $\Rd$ is the union of finitely many shifted copies of a lattice $\L$. We recall that a lattice is discrete co-compact subgroup of $\Rd$ and its density is $1/\vol(\Rd/\L)$ and refer to the textbook of Conway and Sloane \cite{ConSlo98} for an introduction to lattices. The energy of a periodic configuration
$$\Gamma = \bigcup_{k=1}^n (\L + x_k), \quad \{x_1, \ldots , x_n \} \in \Rd/\L, \  x_k \neq x_j, \  k \neq j,$$
is given by
\begin{equation}\label{eq:energy}
    E_g(\Gamma) = \frac{1}{n} \sum_{k=1}^n\sum_{j=1}^n \sum_{\l \in \L \backslash \{x_j-x_k\}} g(|\l+x_j-x_k|).
\end{equation}
So, it is the pairwise interaction of the points under the potential $g$ excluding self-interactions (as the potential may be singular at the origin). We refer to \cite{CohKum07, Coh-Via22} for details on the energy minimization problem and to the textbook of Conway and Sloane \cite{ConSlo98} for an introduction to lattices, packing problems and covering problems as well as to the article of Schuermann and Vallentin \cite{SchVal04}. In \cite{CohKum07} Cohn and Kumar showed that on the real line $\mathbb{R}$ (and at all scales) the scaled integer lattice is universally optimal. They obtained their result by constructing a ``magic function" (using a version of the classical sampling theorem) which proved that the linear programming bounds for the problem (obtained in the same work) are indeed sharp for the scaled integer lattice. An alternative proof, also given in \cite{CohKum07} is via spherical designs. Numerically, the hexagonal lattice also meets the linear programming bound for the energy minimization problem in dimension 2. However, a proof of its universal optimality is still missing. The results are linked to optimal sphere packings and the linear programming bounds for the sphere packing problem obtained by Cohn and Elkies \cite{CohElk03}. In seminal work, the sphere packing problem in dimension 8 was solved by Viazovska \cite{Via17} and in dimension 24 by Cohn, Kumar, Miller, Radchenko, and Viazovska \cite{Coh-Via17}. The problem of energy minimization has also been treated on the sphere $\mathbb{S}^{d-1} \subset \Rd$, which in the case of $d=2$ is a problem of distributing points on the circle line $\mathbb{S}^1 \cong \mathbb{T}$. Often, for general $d \geq 2$, a connection to spherical $t$-designs is given when distributing points on a sphere. We refer to the review by Brauchart and Grabner \cite{BraGra15} and to Hardin and Saff \cite{HarSaf05} for the classical problem of Riesz energy minimization. More recent results on energy minimizing point distributions on spheres were obtained by Beltr\'an and Etayo \cite{BelEta20} or Bilyk, Glazyrin, Matzke, Park, and Vlasiuk \cite{Biletal22}. For spherical $t$-designs we refer to the breakthrough of Bondarenko, Radchenko, and Viazovska \cite{BonRadVia13} and to work of the second author \cite{Stein0} for upper bounds.

\subsection{Polarization problems}
The polarization problem asks to place light sources such that the darkest point has maximal illumination. Often such problems are considered for compact manifolds, such as the sphere. We refer, e.g., to articles, published in different constellations, by Borodachov, Boyvalenkov, Hardin, Reznikov, Saff, and Stoyanova \cite{Bor22, BorHarRezSaf18, BorHarSaf19, BoyDraHarSafSto23}. For more numerical investigations and algorithms we refer to the work by Rolfes, Sch\"uler and Zimmermann \cite{RolSchZim23}. The problem of polarization for Riesz potentials and lattices in $\Rd$ was asked by Saff (cf.\ Problem 1.06 in the collection curated by \textit{American Institute of Mathematics} for the workshop \textit{Discrete Geometry and Automorphic Forms} \cite{AIM}).  We note that many physically important potentials, such as the Riesz potential, can be written as a Laplace transform of a non-negative measure $\mu$. More precisely, any completely monotone function $f: \R_+ \to \R_+$, meaning $(-1)^kf^{(k)}(x) \geq 0$, $\forall k \geq 1$, is the Laplace transform of a non-negative Borel measure as a consequence of the Bernstein--Widder theorem \cite{bernstein, Wid41} (see also the textbook of Schilling, Song, and Vondracek \cite[Chap.~1]{SchSonVon10}). Some results on polarization on $\mathbb{S}^1$ for sufficiently fast decaying and convex potentials have been obtained in \cite[Chap.~14.3] {BorHarSaf19}.
We remark that the Gaussian potential does not fall into the class of completely monotone functions as it is not convex. However, by adjusting distance to squared distance, we get completely monotone functions of squared distance, i.e., $r \mapsto g(r^2)$ where $g$ is completely monotone (compare \cite{Coh-Via22}):
\begin{equation}
    g(r) = \int_0^\infty e^{-\alpha r} \, d\mu(\alpha).
\end{equation}
As remarked in \cite[Sec.~1.2]{Coh-Via22}, it may seem more natural to take completely monotone functions of distance, rather than squared distance, but using squared distance allows for the use of the Gaussian function. In fact, one can check that any completely monotone function of distance is also a completely monotone function of squared distance. We refer to \cite[Sec.~1.2]{Coh-Via22} for this fact and more details. As an example we name the Riesz potentials, also known as inverse power laws, which are obtained as (compare again, e.g., \cite{Coh-Via22})
\begin{equation}
    \frac{1}{r^s} = \int_0^\infty e^{-\alpha r^2} \frac{\alpha^{s/2-1}}{\Gamma(s/2)} \, d\alpha.
\end{equation}
If our result were to hold for all $\alpha > 0$ (when $n$ is fixed), one would immediately have a corresponding result for Riesz potentials as well as the whole class of completely monotone functions of squared distance (given sufficiently fast decay):
\begin{equation}
    \sum_{\gamma \in \Gamma} \left( \int_{0}^\infty e^{-\alpha \gamma^2} \, d\mu(\alpha) \right) = \int_0^\infty \left(\sum_{\gamma \in \Gamma} e^{-\alpha \gamma^2} \right) \, d\mu(\alpha).
\end{equation}

\subsection{Lattices in \texorpdfstring{$\mathbb{R}^2$}{R2}}
Despite the seminal work of Cohn, Kumar, Miller, Radchenko and Viazovska \cite{Coh-Via22} and overwhelming numerical evidence, the universal optimality of the hexagonal lattice, also known as $A_2$ root lattice or sometimes triangular lattice, is still open to date. The best available result is due to Montgomery \cite{Mon88} and states that the hexagonal lattice is optimal among lattices at all scales. More recently, the polarization problem among 2-dimensional lattices has been solved by the authors in joint work with B\'etermin \cite{BetFauSte21}. Local optimality of the hexagonal lattice for lattice polarization and certain potential functions has been derived by the authors in \cite{FauSte19}. In \cite{BetFau23}, B\'etermin and the first author showed that the hexagonal lattice maximizes \textit{Madelung-like} lattice energies (lattice points have alternating signs). This result is somewhat in-between the result of Montgomery \cite{Mon88} and the joint result of the authors with B\'etermin \cite{BetFauSte21} as it does neither clearly relate to sphere packing nor covering. Related results concerning the Lennard--Jones potential (see B\'etermin and Zhang \cite{BetZha15}), which is $r \mapsto r^{-12} - 2 r^{-6}$ and neither non-negative nor monotonic nor convex, show that for different densities different geometrical arrangements can be optimal. This phenomenon is widely called \textit{phase transition}. Some physically relevant consequences of the conjectured universal optimality of the hexagonal lattice (and proven optimality of $\mathsf{E}_8$ and Leech lattice) are discussed by Petrache and Serfaty \cite{PetSer20}. A general survey is given by Lewin and Blanc \cite{lewin}.

\subsection{Heat Equation Sampling}
Our result solves the following problem on $\mathbb{S}^1$ as a byproduct. The problem was originally discussed by Pausinger and the second author \cite{florian} on $\mathbb{T}^2$. Suppose there is an unknown distribution of heat $f \in L^1(\mathbb{S}^1)$ and we are interested in estimating the total heat $\int_{\mathbb{S}^1}f(x)dx$. If the function $f$ is only in $L^1$ then no effective sampling strategies are possible. If we now assume, however, that some time $t>0$ has passed, then the solution of the heat equation $e^{t\Delta}f$ with $f$ as initial condition satisfies
$$ \int_{\mathbb{S}^1} f(x) dx = \int_{\mathbb{S}^1} \left[e^{t\Delta}f\right](x) dx$$
and is also a more regular function for which sampling strategies should be possible.

\begin{corollary}
For any $t>0$ and all $n \geq N(t)$ sufficiently large (depending only on $t$) the worst case sampling error
$$ \sup_{f \in L^1(\mathbb{S}^1)}    \frac{1}{\|f\|_{L^1}} \left|   \frac{1}{n} \sum_{k=1}^{n} \left[ e^{t\Delta} f \right](x_k) - \int_{\mathbb{S}^1} f(x) dx \right| $$
is minimized if and only if the sampling points $\left\{x_1, \dots, x_n\right\}$ are equispaced.
\end{corollary}
\begin{proof}
Interpreting the solution of the heat equation as a Fourier multiplier, 
    \begin{align*}
           \frac{1}{n} \sum_{k=1}^{n} \left[ e^{t\Delta} f \right](x_k) &= \left\langle e^{t\Delta} f, \frac{1}{n} \sum_{k=1}^{n} \delta_{x_k} \right\rangle   = \left\langle  f,  \frac{1}{n} \sum_{k=1}^{n} e^{t\Delta} \delta_{x_k} \right\rangle.   
    \end{align*}
The solution of the heat equation started with a Dirac delta is the Jacobi $\theta$-function
$$  \left[e^{t\Delta} \delta_{x_k}\right](x) = \theta_t(x-x_k) $$
and thus
\begin{align*}
    \left\langle  f,  \frac{1}{n} \sum_{k=1}^{n} e^{t\Delta} \delta_{x_k} \right\rangle &=
    \left\langle  f,  \frac{1}{n} \sum_{k=1}^{n} \theta_t(x-x_k) \right\rangle =  \left\langle  f,  1 + \frac{1}{n} \sum_{k=1}^{n} \left(\theta_t(x-x_k) -1\right) \right\rangle\\
    &= \int_{\mathbb{S}^1} f~ dx + \left\langle  f,  \frac{1}{n} \sum_{k=1}^{n} \left( \theta_t(x-x_k) -1 \right)\right\rangle. 
\end{align*}
Using $L^1-L^{\infty}$ duality, we arrive that
$$ \sup_{f \in L^1(\mathbb{S}^1)} \left| \left\langle  f,  \frac{1}{n} \sum_{k=1}^{n} \left( \theta_t(x-x_k) -1 \right)\right\rangle \right| =  \left\|\frac{1}{n} \sum_{k=1}^{n} \left( \theta_t(x-x_k) -1 \right) \right\|_{L^{\infty}}.$$
Our results show that the maximum is minimized and the minimum is maximized if and only if the points are equispaced. This implies the statement.
\end{proof}
\begin{remark}
    It was pointed out to us by one of the referees that we can drop the condition $n \geq N(t)$ as  Proposition \ref{pro:Cohn-Kumar} is sufficient in order to prove the above corollary for any $t$ and all $n$. The argument goes along the same lines as above, but then continues in the following way. We need to show that
    \begin{equation}\label{eq:remark}
        \left\| \sum_{k=1}^n \theta_t(x-x_k) - n \right\|_{L^\infty} \geq \left\| \sum_{k=1}^n \theta_t\left(x-\frac{k}{n}\right) - n \right\|_{L^\infty}
    \end{equation}
    Using a trivial estimate and then Proposition \ref{pro:Cohn-Kumar} we get
    \begin{equation}
        \left\| \sum_{k=1}^n \theta_t(x-x_k) - n \right\|_{L^\infty} \hspace{-7pt}
        \geq \max_x \left( \sum_{k=1}^n \theta_t(x-x_k) - n \right) \geq \max_x \left( \sum_{k=1}^n \theta_t\left(x-\frac{k}{n}\right) - n \right)
    \end{equation}
    In order to show \eqref{eq:remark} it now suffices to show that
    \begin{align}
        & \max_x \sum_{k=1}^n \theta_t\left(x-\frac{k}{n}\right) - n \geq n - \min_x \sum_{k=1}^n \theta_t\left(x-\frac{k}{n}\right)\\
        \Longleftrightarrow \qquad
        & \max_x \sum_{k=1}^n \theta_t\left(x-\frac{k}{n}\right) + \min_x \sum_{k=1}^n \theta_t\left(x-\frac{k}{n}\right) \geq 2n.
    \end{align}
    It is a remarkable property of the theta function (cf.~\cite[Chap.~4, eq.~(22)]{ConSlo98}) that 
    \begin{equation}
        \max_x \sum_{k=1}^n \theta_t\left(x-\frac{k}{n}\right) + \min_x \sum_{k=1}^n \theta_t\left(x-\frac{k}{n}\right) = \max_x \sum_{k=1}^{2n} \theta_t\left(x-\frac{k}{2n}\right) \geq 2n
    \end{equation}
    The details of this argument are provided in \S \ref{sec:notation} and \S \ref{sec:proof_Cohn-Kumar}. The crucial property is that, in the equispaced case, the minimum is achieved exactly midway between the points and the maximum at the points themselves.
\end{remark}

\subsection{Shift invariant systems}
A shift invariant system $V^2(g)$ on $\R$ with a generator $g \in \Lt[]$ is a space of functions of the form
\begin{equation}
    V^2(g) = \{ f(x) = \sum_{k \in \Z} c_k \, g(x-k) \mid (c_k) \in \ell^2(\Z) \}.
\end{equation}
An example is the classical Paley--Wiener space $PW(\R)$ of band-limited functions, i.e., $\supp(\widehat{f}) \subset [-1/2,1/2]$, which is generated by $\sinc(x) = \sin(\pi x)/(\pi x)$.
For a set $\Gamma \subset \R$, we say that it is a set of sampling for $V^2(g)$ if and only if there exist positive constants $0<A\leq B<\infty$, depending on $g$ and $\Gamma$, such that
\begin{equation}
    A \| f \|_{\Lt[]}^2 \leq \sum_{\gamma \in \Gamma} |f(\gamma)|^2 \leq B \|f\|_{\Lt[]}^2, \quad \forall f \in V^2(g).
\end{equation}
For the motivation of (non-uniform) sampling in $V^2$ we refer to the article by Aldroubi and Gröchenig \cite{AldGro01}. Characterizing sampling sets for given generator $g$ is a very difficult problem. A necessary condition is that the (lower Beurling) density of the set is at least 1. The case of density 1 is referred to as critical sampling. For a large class of functions, including the Gaussian function $x \mapsto e^{-\alpha x^2}$, $\alpha > 0$, the problem was solved by Gr\"ochenig, Romero, and St\"ockler \cite{GroRomSto18}. The case of critical sampling with Gaussian generator is treated by Baranov, Belov, and Gr\"ochenig \cite{BarBelGro22}.

Our results suggest that for the space $V^2(\phi_\alpha)$, where $\phi_\alpha$ is a Gaussian, the bound $B$ is minimal and $A$ is maximal for equispaced sampling. Lastly, we mention the relatively new area of dynamical sampling introduced by Aldroubi, Cabrelli, Molter, and Tang \cite{AldCabMolTan17}. This combines the sampling problem with dynamical systems. In particular, we find connections between the heat equation and the sampling problem, as described by Aldroubi, Gr\"ochenig, Huang, Jaming, Krishtal, and Romero~\cite{Aldetal21}. Ulanovskii and Zlotnikov \cite{UlaZlo21} described sampling sets for $PW(\R)$ so that $f$ can be reconstructed from samples of $f*\varphi_t$, where $\varphi_t$ is a convolution kernel of a dynamical process. It would be interesting to see how our results connect to this area.

\section{Notation and remarks}\label{sec:notation}
\subsection{Basic notation}
To clarify normalization, we note that we use the following version of the Fourier transform of a suitable function $f$ on the real line:
\begin{equation}\label{eq:FT}
    \widehat{f}(\omega) = \int_{\R} f(x) e^{-2 \pi i \omega x} \, dx,
    \qquad \text{ so } \qquad
    \|f\|_{\Lt[]} = \|\widehat{f}\|_{\Lt[]}.
\end{equation}
Thus, the Poisson Summation Formula reads (see, e.g., Gr\"ochenig \cite[Chap.~1.4]{Gro01})
\begin{equation}
    \sum_{k \in \Z} f(k+x) = \sum_{\ell \in \Z} \widehat{f}(\ell) \, e^{2 \pi i \ell x}.
\end{equation}
The Fourier transform of a Gaussian is another Gaussian, differently scaled (see, e.g., Folland \cite[App.~A]{Fol89}):
\begin{equation}
    \text{if} \quad
    \phi_\alpha(x) = e^{-\pi \alpha x^2}, \; \alpha > 0,
    \quad \text{then} \quad
    \widehat{\phi_\alpha}(\omega) = (1/\sqrt{\alpha}) \ e^{-(\pi/\alpha) \omega^2} = (1/\sqrt{\alpha}) \ \phi_{1/\alpha}(\omega).
\end{equation}
The periodization of $\phi_\alpha$ will be called a periodic Gaussian: $\sum_{k \in \Z} \phi_{\alpha}(x+k)$.
A periodic configuration $\Gamma \subset \R$ with period $\delta$ is a set of points of the following form:
\begin{equation}
    \Gamma = \bigcup_{k=1}^n (\L + x_k),
    \quad \text{ where } \quad
    \L = \delta \Z , \, \delta > 0,
    \quad
    x_k \neq x_j, \, k \neq j,
    \quad
    x_k \in [0, \delta).
\end{equation}
The density $\rho$ of a configuration $\Gamma$ is the number $n$ of points per period 
$\rho = n/\delta.$

\subsection{Polarization on the real line}
We are now interested in the following polarization problem: which periodic configuration of fixed density $\rho$ maximizes
\begin{equation}
    \min_x \frac{1}{\sqrt{\alpha}}~ \sum_{\gamma \in \Gamma} \phi_{1/\alpha}(x-\gamma) \; ?
\end{equation}
We quickly note that, fixing the amounts of points per period, a minimizer always exists by compactness. We call the above quantity the polarization of $\Gamma$ and seek to find the maximal polarization. In general, the minimum depends on $\Gamma$ and its density $\rho$ as well as on $\alpha$. For equidistributed points, however, the minimum is always achieved midway between successive points (as we will prove as part of the proof of the main result).
The polarization may more explicitly be written in one of the following ways:
\begin{align}
    \min_x \frac{1}{\sqrt{\alpha}} \sum_{\gamma \in \Gamma} \phi_{1/\alpha}(x-\gamma)
    & = \min_x \frac{1}{\sqrt{\alpha}} \sum_{j=1}^n \sum_{k \in \Z} e^{-\pi \frac{\delta^2}{\alpha} \left(k+\frac{x_j-x}{\delta}\right)^2}\label{eq:polar_1}\\
    & = \min_x \frac{1}{\delta} \sum_{j=1}^n \sum_{k \in \Z} e^{- \pi \frac{\alpha}{\delta^2} k^2} e^{2 \pi i k \frac{(x_j-x)}{\delta}}\label{eq:polar_2},
\end{align}
where the second equality is due to the Poisson Summation Formula. Note that in this explicit formula $\{x_1, \ldots , x_n\} \subset [0,\delta)$. By identification of a configuration $\Gamma$ with $(x_1, \ldots , x_n) \in (\delta \mathbb{T})^n$ we see that a maximizing configuration must exist by compactness. Clearly, neither the factor $1/\sqrt{\alpha}$ nor the factor $1/\delta$ are of relevance for the minimization process or determination of the maximizing configuration. We will next show that for \eqref{eq:polar_1} and any fixed $n$, $\alpha > 0$, and $\delta > 0$ there is always an equivalent problem with the same $n$, $\delta = 1$ and different $\alpha$. For $\Gamma = \bigcup_{j=1}^n (\delta \Z + x_j)$, $x \in [0,\delta)$, $x_j \in [0,\delta)$ we simply write
\begin{align}
    \sum_{j=1}^n \sum_{k \in \Z} e^{- \pi \alpha \left(\delta k + x_j - x\right)^2}
    & = \sum_{j=1}^n \sum_{k \in \Z} e^{-\pi \widetilde{\alpha} \left( k + \widetilde{x}_j - \widetilde{x} \right)^2},
\end{align}
where $\widetilde{x} = x/\delta \in [0,1)$, $\widetilde{x}_j = x_j/\delta \in [0,1)$ and $\widetilde{\alpha} = \alpha \delta^2$. We see that we may thus assume that the points $\{x_1, \ldots , x_n\}$ are distributed in $[0,1)$ and that $\Gamma$ is 1-periodic (and of density $n$). Using the Poisson Summation Formula we see that finding the optimal configuration for \eqref{eq:polar_2} is the same as maximizing
\begin{equation}
    \min_x \sum_{j=1}^n \sum_{k \in \Z} e^{-\pi \alpha k^2}e^{2 \pi i k (x_j-x)} = \min_x \sum_{j=1}^n \theta_\alpha(x_j-x).
\end{equation}
This is (up to flipping the argument) exactly the quantity $f_\alpha(x)$ from \eqref{eq:f_alpha} considered in our main result. Note that by the Poisson Summation Formula
\begin{equation}
    f_\alpha(x) = \frac{1}{\sqrt{\alpha}} \, p_{1/\alpha}(x).
\end{equation}

\subsection{Theta functions}
The problem can be written as a variational problem for a finite superposition of real-valued theta functions. For parameter $\tau \in \mathbb{H}$ (complex upper half-plane) and argument $z \in \C$ the classical theta function is
\begin{equation}
    \vartheta(z;\tau) = \sum_{k\in \Z} e^{\pi i \tau k^2}e^{2 \pi i k z}.
\end{equation}
This function is holomorphic in $\tau$ and entire in $z$. For $\tau = i \alpha$, $\alpha > 0$ and $z = x \in \R$ the function becomes real-valued and we use the notation:
\begin{equation}
	\theta(x;\alpha) = \sum_{k \in \Z} e^{-\pi \alpha k^2} e^{2 \pi i k x} = \sum_{k \in \Z} e^{-\pi \alpha k^2} \cos(2 \pi  k x) = 1 + 2 \sum_{k \geq 1} e^{-\pi \alpha k^2} \cos(2 \pi k x).
\end{equation}
Note that the function $\theta(x;\alpha)$ is the heat kernel on the flat torus $\R/\Z$. As such it has mean value 1, which is easily verified by a small computation;
\begin{equation}
    \int_0^1 \sum_{k \in \Z} e^{-\pi \alpha k^2} e^{2 \pi i k x} \, dx = \sum_{k \in \Z} e^{-\pi \alpha k^2} \int_0^1  e^{2 \pi i k x} \, dx = \sum_{k \in \Z} e^{-\pi \alpha k^2} \delta_{k,0} = 1,
\end{equation}
where $\delta_{k,0}$ is the Kronecker delta. The function $\vartheta(z;\tau)$ and, hence, $\theta(x;\alpha)$ can be expressed as an infinite product known as the Jacobi triple product, which is a special case of the Macdonald identities for affine root systems~\cite{Mac71}:
\begin{align}
	\vartheta(z;\tau) & = \prod_{k \geq 1} \left(1 - e^{2 k \pi i \tau} \right) \left(1 + e^{(2k -1) \pi i \tau} e^{2 \pi i z} \right) \left(1 + e^{(2k -1) \pi i \tau} e^{-2 \pi i z} \right)\\
	& = \prod_{k \geq 1}  \left(1 - e^{2 k \pi i \tau} \right) \left(1 + 2 \cos(2 \pi z) e^{(2k-1)\pi i \tau} + e^{2 (2k-1) \pi i \tau} \right).
\end{align}
We refer to textbooks of Mumford \cite{Mum_Tata_I}, Stein and Shakarchi \cite{SteSha03}, or Whittaker and Watson \cite{WhiWat69} for more details on elliptic functions.

\section{Proof of Proposition \ref{pro:Cohn-Kumar}}\label{sec:proof_Cohn-Kumar}
Proposition \ref{pro:Cohn-Kumar} follows relatively easily from the work of Cohn and Kumar \cite{CohKum07} and the Poisson Summation Formula. The heart of the argument has three ingredients:
\begin{enumerate}[(1)]
    \item first, universal optimality shows that, for any fixed $\alpha > 0$, the interaction energy
    $$ \frac{1}{n} \sum_{k,j=1}^n  \theta_\alpha(x_j-x_k) \qquad \mbox{is minimized for equispaced points.}$$
    \item The second ingredient is a trivial estimate that arises from replacing an average (arithmetic mean) of values by its maximum
    \begin{equation}\label{eq:max-mean}
        \max_{x}  \sum_{k=1}^n  \theta_\alpha(x-x_k) \geq  \frac{1}{n} \sum_{j=1}^n \sum_{k=1}^n  \theta_\alpha(x_j-x_k).
    \end{equation}
    \item The third ingredient is that (2) is \textit{sharp} whenever the points are equispaced (which, simultaneously, by universal optimality, minimizes the lower bound in (2) just above). There is a magic ingredient where, for equispaced points, the maximum of $\sum_{k=1}^{n} \theta_{\alpha}(x-x_k)$ is attained at the points $x_j$ themselves.
\end{enumerate}
We remark that the counterpart to (1) is false for the minimization problem. Likewise, regarding (3), the location of the minimum depends in a highly nonlinear fashion on the location of the points. Understanding the minimum and the considered polarization problem thus requires a different approach.

\begin{proof}
(1) We note that the energy for the potential $\Phi = 1/\sqrt{\alpha} \ \phi_{1/\alpha}$ is given by
\begin{equation}\label{eq:energy_theta}
    E_{\Phi}(\Gamma) = \frac{1}{n} \sum_{k=1}^n\sum_{j=1}^n \sum_{\ell \in \Z} \frac{1}{\sqrt{\alpha}}\phi_{1/\alpha}(\ell+x_j-x_k) = \frac{1}{n} \sum_{k=1}^n \sum_{j=1}^n \theta_\alpha(x_j-x_k),
\end{equation}
where the second equality comes from the Poisson Summation Formula. The potentials are sitting on the periodic configuration $\Gamma$. However, not only their sum is considered but all their pairwise interactions and the sum over all of them. The condition $\l \in \L \backslash \{x_k-x_j\}$ in \eqref{eq:energy} excludes self-interaction as the potential function $g$ is allowed to be singular at 0 (this is also of physical relevance).
For the Gaussian, we may allow self-interaction (which adds a fixed additive constant determined by normalization, but independent of $\Gamma$) and we do not need to exclude it. If $\Gamma_0 = \bigcup_{j=1}^n (\Z + (j-1)/n) = (1/n) \Z$, then the energy can be written as (after applying the Poisson Summation Formula)
\begin{equation}
    E_{\Phi}(\Gamma_0) = \frac{1}{n} \sum_{k=1}^n \sum_{j=1}^n \theta_\alpha \left(\ell+\frac{j-k}{n} \right) = \frac{1}{n} \sum_{k=1}^n \sum_{j=1}^n \theta_\alpha \left(\ell+\frac{j}{n} \right) = \sum_{j=1}^n \theta_\alpha \left(\ell+\frac{j}{n} \right),
\end{equation}
where the second and third equalities are due to the periodicity of $\theta_\alpha$. The universal optimality of the (scaled) integers due to Cohn and Kumar \cite{CohKum07} states, for all $\alpha > 0$,
\begin{equation}\label{eq:Cohn-Kumar}
    E_\Phi(\Gamma_0) \leq E_\Phi(\Gamma)
    \quad \text{ with equality if and only if } \quad
    \Gamma = \Gamma_0 + z, \ z \in \R.
\end{equation}
Note that the result in \cite{CohKum07} as well as ours also hold for arbitrary scaling.

(2) is a trivial observation and does not require any more details.

(3) For $\Gamma_0$ the maxima of $p_\alpha$ (or likewise $f_\alpha$) are attained at the equispaced points $\{0, 1/n, \ldots , (n-1)/n\}$ (compare Proposition \ref{pro:minimum_1/2}). This follows by a simple application of the Poisson Summation Formula and the triangle inequality.  This allows for various additional tools to be used, in particular, it allows for a lossless application of the triangle inequality. We give the proof for the integers $\Z$ but the proof can easily be adjusted to scaled integers $\delta \Z$ (replace $k$ by $k/\delta$ and adjust the Poisson Summation Formula accordingly). We perform the following small computation:
\begin{align}\label{eq:maximum_equispaced}
    \frac{1}{\sqrt{\alpha}} \, p_{1/\alpha}(x) & =  f_{\alpha}(x) = \sum_{k \in \Z} e^{-\pi \alpha k^2} e^{2 \pi i k x} \leq \sum_{k \in \Z} e^{-\pi \alpha k^2} \left|e^{2 \pi i k x}\right|\\
    & = \sum_{k \in \Z} e^{-\pi \alpha k^2} = f_{\alpha}(0) = \frac{1}{\sqrt{\alpha}} \, p_{1/\alpha}(0),
    \qquad \forall \alpha > 0.
\end{align}
So, the maximum is attained at 0 and by periodicity at all points in $\Z$ (or $\delta \Z$).
\begin{figure}[h!t]
    \includegraphics[width=.95\textwidth]{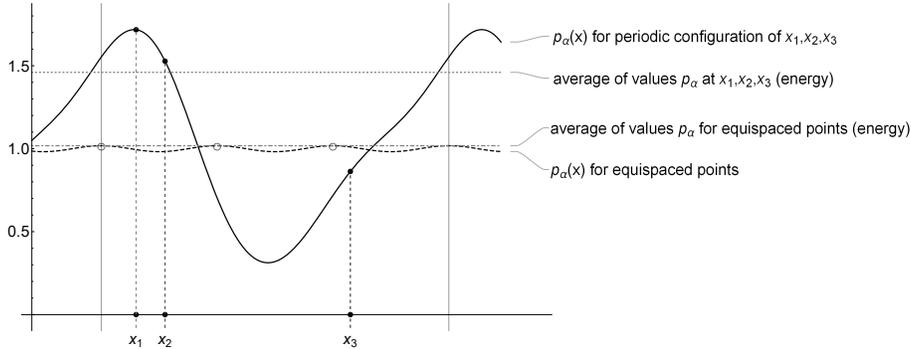}
    \captionsetup{width=0.99\linewidth}
    \caption{Illustration of the result of Cohn and Kumar \cite{CohKum07}. Building the average of $p_\alpha(x)$ at the points $\{x_1, \ldots , x_n\}$ (in this case $n=3$) for periodic, non-equispaced configuration always yields a larger value than for the equispaced points. As we sum $n$ times the maximum in the equispaced case, it follows that the maximum of $p_\alpha(x)$ is minimal only for the equispaced configuration.}
    \label{fig:energy}
    \vspace{-11pt}
\end{figure}

Note that $E_\Phi(\Gamma)$ builds the average of all values taken on $\Gamma$ (see Figure \ref{fig:energy}).  Now recall that \eqref{eq:maximum_equispaced} tells us that for the equispaced configuration $\Gamma_0$ the maximum is attained exactly on $\Gamma_0$. It readily follows from \eqref{eq:max-mean}, \eqref{eq:energy_theta} and \eqref{eq:Cohn-Kumar} that
\begin{equation}
    \max_{x} f_\alpha(x) = \max_{x} \sum_{j=1}^n \theta_\alpha \left(x+x_j\right)
    \quad \text{ is minimal if and only if $\Gamma$ is equispaced.}
\end{equation}
This gives Propositions \ref{pro:Cohn-Kumar} as a simple consequence of the result in \cite{CohKum07}. \end{proof}

\section{Proof of the Main Result}\label{sec:proof_main}
We start with an overall overview of the argument. It is fairly modular and the subsections reflect its overall structure. We also emphasize that, due to the fast decay of the Fourier coefficients, the argument is somewhat forgiving when it comes to polynomial estimates in the number of points. As a consequence, some of the subsequent proofs are given in its simplest rather than their optimal form. The main argument comes in two parts: the first part shows that optimizing configurations have to be exponentially close (in $n$) to the equispaced distribution. The structure of the first part is as follows.
\begin{enumerate}
    \item \S \ref{sec:minimizer} uses some basic facts about theta-functions. We show that if the points are equispaced, then the minimum is attained exactly at the midpoints between the equispaced points. This then allows us to deduce
    $$   \min_x \sum_{j=1}^n  \theta_{\alpha}\left(x - \frac{j}{n}\right) = n - 2 n e^{-\pi \alpha n^2}  + \mathcal{O}(n e^{-4 \pi \alpha n^2})$$
    which already shows some of the difficulty: the difference between the average and the minimum can be super-exponentially small in $n$.
    
    \item \S \ref{sec:L1-estimate} introduces a trivial $L^1$-estimate (essentially pigeonholing) and a nontrivial estimate: the McGehee--Pigno--Smith inequality \cite{mcgehee}, and independently discovered by Konyagin \cite{Kon82}. It was pointed out to us by an anonymous referee that the McGehee--Pigno--Smith inequality can be avoided and we present this more elementary argument as well.
    
    \item\label{item:estimate} \S \ref{sec:small_Fourier_coeff} combines these ingredients to prove that if $\left\{x_1, \dots, x_n\right\} \subset [0,1)$ is an optimal configuration (meaning one maximizing the minimum), then the first $n-1$ Fourier coefficients of the measure $\mu = \sum_{j=1}^{n} \delta_{x_j}$ must be small, more precisely
    $$ \max_{1 \leq |k| \leq n-1}  \left| \sum_{j=1}^{n} e^{-2 \pi i k x_j} \right| \leq 2000 \cdot n^2 \cdot e^{- \pi \alpha (2n-1)}.$$ 
    
    \item We note that for equispaced points the first $n-1$ Fourier coefficients all vanish. \S \ref{sec:regular_gaps} proves a basic estimate, invoking the classical Fej\'er kernel, showing if the first $n-1$ Fourier coefficients of $\mu$ are close to 0, then the $n$ points are (quantitatively) close to $n$ equispaced points. Since the estimate from (\ref{item:estimate}) is extremely small, exponentially small in $n$, we get that any optimal configuration has to be exponentially close to equispaced.
\end{enumerate}

\medskip

The second part of the proof shows that the only configuration that is exponentially close (in $n$) to the equispaced distribution and has maximal polarization is the equispaced distribution: this part can be understood as a detailed analysis of the perturbative regime. The main idea lies in making the ansatz $x_j = j/n + \varepsilon_j$ together with the explicit Fourier series representation
\begin{align*}
    \sum_{j=1}^{n} \theta_\alpha \left(x - \frac{j}{n} + \varepsilon_j \right) =  \sum_{k \in \mathbb{Z}} e^{-\pi \alpha k^2} \left( \sum_{j=1}^{n}  e^{2 \pi i k \varepsilon_j } e^{-2 \pi i k  \frac{j}{n}} \right) e^{2 \pi i k x}.
\end{align*}
Since the problem is invariant under shifts, we can (and have to) assume that $\varepsilon_1 + \dots + \varepsilon_n = 0$ to eliminate the invariance of the problem under translation. The argument is then structured as follows.

\begin{enumerate}
    \setcounter{enumi}{4}
    
    \item  In \S \ref{sec:main_contribution} we show that the frequencies where $k$ is a multiple of $n$ are exactly the terms that contribute when the points are equispaced: among these frequencies only $k \in \left\{-n,0,n\right\}$ have a sizeable contribution, the rest is small. The equispaced points yield $n$ local minima and our goal is to show that at least one of these minima further decreases unless $\varepsilon_j = 0$ for all $1 \leq j \leq n$ (meaning the points are equispaced again).

    \item We consider the trigonometric polynomial $g_1(x)$ which is the restriction to the first $(n-1)/2$ frequencies. By a modified Poincar\'e inequality, we will prove in \S \ref{sec:trig_poly} that any such trigonometric polynomial assumes a small negative value at at least one of the points of the form $(k+1/2)/n$, for $0 \leq k \leq n-1$.  This negative contribution is going to make at least one of the minima much smaller. It remains to make sure that this cannot be counteracted by contributions coming from the other frequencies.
    
    \item  There are two remaining parts to analyze: $g_2(x)$ defined by restricting summation to the frequencies $n/2 \leq |k| \leq n-1$ and $h(x)$ for all the remaining frequencies. We will prove in \S \ref{sec:outline_g2_h} that $\|g_2\|_{L^{\infty}}, \|h\|_{L^{\infty}} \ll \|g_1\|_{L^2}$. Indeed, these terms are many orders of magnitude smaller.
    
    \item The main ingredient for showing the last step is a surprising appearance of the \textit{Discrete Fourier Transform} (see \S \ref{sec:DFT}) hidden in the Fourier coefficients: since the sum of the perturbations $\varepsilon_1 + \dots + \varepsilon_n = 0$, we can approximate the Fourier coefficients whenever $k$ is not a multiple of $n$, as
    \begin{align*}
        \sum_{j=1}^{n}  e^{2 \pi i k \varepsilon_j } e^{-2 \pi i k  \frac{j}{n} } &= 2\pi i k  \sum_{j=1}^{n}  \varepsilon_j  e^{-2 \pi i k  \frac{j}{n} } + \mathcal{O}\left(k^2 \sum_{j=1}^{n} \varepsilon_j^2 \right),
    \end{align*}
    where the sum is merely a Discrete Fourier Transform of the $\varepsilon_1, \dots, \varepsilon_n$. This allows us to deduce a certain type of symmetry (because the $\varepsilon_j$ are real-valued) which will be used to prove $\|g_2\|_{L^{\infty}} \ll \|g_1\|_{L^2}$. It also guarantees that not all Fourier coefficients are small (via a Plancherel identity).
    
    \item The final inequality, established in \S \ref{sec:final_estimate}, is, assuming the perturbations $\varepsilon_j$ are exponentially close to 0, that the minimum
    $$ Z =  \min_{0 \leq k \leq n-1} \sum_{j=1}^{n} \theta_\alpha \left( \frac{k+1/2}{n} - \frac{j}{n} + \varepsilon_j \right)$$
    satisfies 
    \begin{align*}
        Z  \leq   \left[\min_{0 \leq x \leq 1} \sum_{j=1}^{n} \theta_\alpha \left(x - \frac{j}{n}  \right) \right]  - C e^{-\pi \alpha \left( \frac{n-1}{2} \right)^2} \left(\sum_{j=1}^{n} \varepsilon_j^2 \right)^{1/2}
    \end{align*}
    which then forces all the perturbations to vanish.
\end{enumerate}

\section*{Part 1 of the proof}
\subsection{Minimizer for equidistributed points}\label{sec:minimizer}
We first prove that for equispaced points the minimum is attained exactly midway between two subsequent points. It is somewhat remarkable, and indicative of the difficulty of the problem, that even this very intuitive statement does not appear to have a very simple proof.

\begin{proposition}\label{pro:minimum_1/2} We have, for all $0 \leq \ell \leq n-1$
$$   \min_{x \in \mathbb{T}}\sum_{j=1}^n  \theta_{\alpha}\left(x - \frac{j-1}{n}\right) = \sum_{j=1}^n  \theta_{\alpha}\left(\frac{\ell+1/2}{n} - \frac{j-1}{n}\right).$$
\end{proposition}
\begin{proof}
    Suppose $\left\{x_1, \dots, x_n\right\} \subset [0,1)$ are equispaced points, $x_j = (j-1)/n$. Then
    $$ \sum_{j=1}^n  \theta_{\alpha}\left(x - x_j\right) = \sum_{j=1}^n \sum_{k \in \Z} e^{-\pi \alpha k^2} e^{2 \pi i k (x-x_j)} = \sum_{k \in \mathbb{Z}} e^{-\pi \alpha k^2} \left( \sum_{j=1}^n e^{-2 \pi i k x_j} \right) e^{2 \pi i k x}.$$
    As the points are equispaced, we have
    $$  \sum_{j=1}^n e^{-2 \pi i k x_j}  = \begin{cases} n \qquad &\mbox{whenever}~n \big| k  \\ 0 \qquad &\mbox{otherwise.} \end{cases}$$
    Therefore
    \begin{equation}\label{eq:theta_min_periodic}
        \sum_{j=1}^n  \theta_{\alpha}\left(x - \frac{j-1}{n}\right) = n\sum_{k \in \mathbb{Z}} e^{-\pi \alpha k^2 n^2}  e^{2 \pi i k n x} = n \cdot \theta (nx; n^2 \alpha).
    \end{equation}
    We use the Jacobi triple product representation of the theta function
    \begin{equation}\label{eq:triple_product}
	\theta(x;\alpha) = \prod_{k \geq 1} \left(1 - e^{-2 k \pi \alpha} \right) \left(1 + 2 \cos(2 \pi x) e^{-(2k-1)\pi \alpha} + e^{-2 (2k-1) \pi \alpha} \right).
    \end{equation}
Only now it is easy to find the minimum: in the product formula of $\theta$ each factor is minimized if and only if $x \in \Z + 1/2$, as the cosine-term is decisive and assumes its minimum there.
The following inequality is an immediate consequence:
\begin{equation}
	n \cdot \theta\left(\frac{1}{2};n^2\alpha\right) \leq n \cdot \theta(n x; n^2 \alpha), \quad \forall \alpha > 0,
\end{equation}
where equality holds if and only if $x \in \frac{1}{n} \left(\Z + \frac{1}{2}\right)$. The result follows from \eqref{eq:theta_min_periodic}.
\end{proof}

\begin{figure}[ht]
    \includegraphics[width=.95\textwidth]{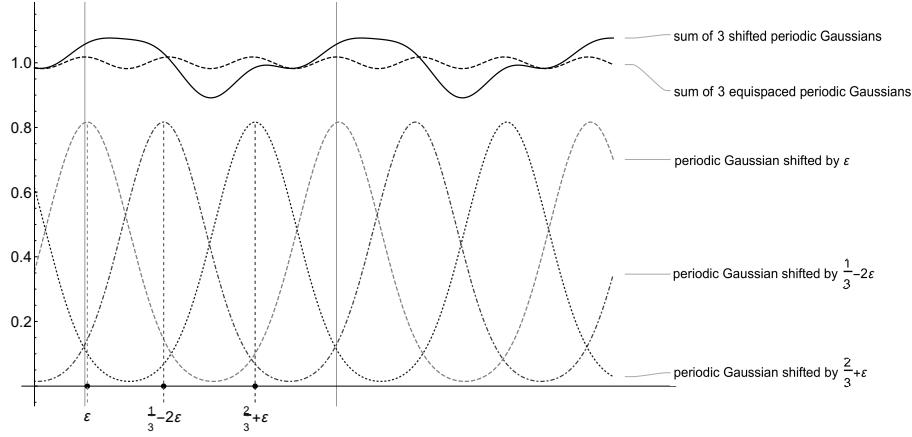}
    \captionsetup{width=.975\linewidth}
    \caption{For the sum of equispaced periodic Gaussians the minimum is achieved midway between successive shifts. For sums of shifts by a general periodic configurations it is rather difficult to grasp the minimum. For the plot we have normalized the sum to oscillate around 1, i.e., the integral over 1 period is 1.}
    \label{fig:Gaussians}
\end{figure}

This fact will be used frequently since it allows for the natural point of comparison (see Figure \ref{fig:Gaussians}). The next step consists in computing the actual size of the minimum. Using, again, the fact that unit roots sum to 0 we end up with
\begin{align*}
     \sum_{j=1}^n  \theta_{\alpha}\left(x - \frac{j}{n}\right) &= n\sum_{k \in \mathbb{Z}} e^{-\pi \alpha k^2 n^2}  e^{2 \pi i k n x} \\
     &=n + 2 n e^{-\pi \alpha n^2} \cos\left( 2 \pi  n x \right)  + \mathcal{O}(n e^{-4 \pi \alpha n^2}).
\end{align*}
Since we know from Proposition \ref{pro:minimum_1/2} that the minimum is attained exactly in the middle between two subsequent points, we have the explicit representation
\begin{align*}
   \min_x \sum_{j=1}^n  \theta_{\alpha}\left(x - \frac{j-1}{n}\right) &= \sum_{j=1}^n  \theta_{\alpha}\left(\frac{1}{2n} - \frac{j-1}{n}\right) \\
   &= n\sum_{\ell \in \mathbb{Z}} e^{-\pi \alpha \ell^2 n^2}  e^{2 \pi i \ell n \frac{1}{2n}} = n + 2n \sum_{\ell=1}^{\infty} (-1)^{\ell} e^{-\pi \alpha \ell^2 n^2} \\
   &= n - 2 n e^{-\pi \alpha n^2}  + \mathcal{O}(n e^{-4 \pi \alpha n^2}).
\end{align*}

\subsection{\texorpdfstring{$L^1$}{L1}-estimates}\label{sec:L1-estimate}
We continue with a basic $L^1$-estimate and a not so basic $L^1$-estimate. The reason why $L^1$ is a natural space to bound deviation from the mean is given by the following elementary pigeonhole argument.

\begin{lemma} \label{lem:basic}
Suppose $g:[0,1] \rightarrow \mathbb{R}$ is a periodic, continuous function with mean value 0. Then
$$ \min_{0 \leq x \leq 1} g(x) \leq  - \frac{\|g\|_{L^1}}{2}.$$
\end{lemma}
\begin{proof}
Since $g$ has mean value 0, we have 
$$ \int_0^1 \max(0,g(x)) dx =  - \int_0^1 \min(0, g(x)) dx$$
and
thus
$$ \int_0^1 \min(0,g(x)) dx = - \frac{\|g\|_{L^1}}{2}.$$
The argument then follows from
$$ - \frac{\|g\|_{L^1}}{2} = \int_0^1 \min(0,g(x)) dx \geq \min_{0 \leq x \leq 1} g(x). $$
\end{proof}

We also use an inequality discovered independently by McGehee, Pigno, and Smith \cite{mcgehee} and Konyagin \cite{Kon82}. It arose in their solutions of the Littlewood conjecture. 

\begin{thm}[McGehee, Pigno, Smith \cite{mcgehee}]
    For any set of integers $\lambda_1 < \lambda_2 < \dots < \lambda_n$ we have
    $$ \int_0^{1} \left| \sum_{j=1}^{n} a_j e^{ 2\pi i \lambda_j t} \right| dt \geq \frac{1}{200} \sum_{j=1}^{n} \frac{|a_j|}{j}.$$
\end{thm}
We note that Konyagin \cite{Kon82} did not explicitly provide the constant. McGehee, Pigno, and Smith work over the interval $[0,2\pi]$ and show that the inequality holds with constant $c = 1/30$ which leads to $1/(60 \pi) \geq 1/200$ being an admissible constant when working over the interval $[0,1]$.  Stegeman \cite{steg} showed that one can take $c = 4/\pi^3$ on $[0,2\pi]$ which would lead to a constant of $1/50$ being admissible after rescaling to $[0,1]$. In any case, the precise value of the constant will not be of importance for the subsequent argument.
We will use the McGehee--Pigno--Smith inequality to derive a lower bound on the $L^1$-norm of the deviation of the sum of Jacobi $\theta$-functions from their mean. We note that if the lower bound is large, then the $L^1$-norm is large and, as a consequence, the minimal value attained by the function has to be quite a bit smaller than its average. Since we want to avoid this, this will implicitly force the first few Fourier coefficients to be small. It has been pointed out by an anonymous referee that, for the purposes of our argument, the McGehee--Pigno--Smith inequality can be avoided as follows: we have, for any $1 \leq k \leq n$ that
\begin{align*}
      \int_0^{1} \left| \sum_{j=1}^{n} a_j e^{ 2\pi i \lambda_j t} \right| dt &= 
      \int_0^{1} \left| e^{-2 \pi i\lambda_k t} \sum_{j=1}^{n} a_j e^{ 2\pi i \lambda_j t} \right| dt = \int_0^{1} \left|  \sum_{j=1}^{n} a_j e^{ 2\pi i (\lambda_j - \lambda_k) t} \right| dt \\
      &\geq   \left| \int_0^{1}  \sum_{j=1}^{n} a_j e^{ 2\pi i (\lambda_j - \lambda_k) t} dt \right| = |a_k|
\end{align*}
and therefore also
$$  \int_0^{1} \left| \sum_{j=1}^{n} a_j e^{ 2\pi i \lambda_j t} \right| dt \geq \frac{1}{n} \sum_{j=1}^{n} |a_j|.$$
This estimate is indeed sufficient for the remainder of the argument. This is partially due to the fact that the multipliers in the Fourier series decay extremely rapidly (i.e. like a Gaussian). Using the McGehee--Pigno--Smith inequality instead of the more elementary inequality might prove advantageous when trying to establish an analogous result with a kernel whose Fourier transform decays more slowly. Using the McGehee--Pigno--Smith or the more elementary inequality gives the following.

\begin{lemma}\label{lem:McGPS_estimate}
    We have, for all $\left\{x_1, \dots, x_n \right\} \subset [0,1)$ that
    $$ \left\|  \sum_{j=1}^n  \theta_{\alpha}\left(x - x_j\right)  - n\right\|_{L^1} \geq \frac{1}{400n}  \sum_{\substack{k \neq 0 \\ |k| \leq n}}  e^{-\pi \alpha k^2} \left| \sum_{j=1}^{n} e^{-2 \pi i k x_j} \right| - \mathcal{O}(n  e^{-\pi \alpha (n+1)^2}).$$
\end{lemma}
\begin{proof}
Our object of interest
$$   \sum_{j=1}^n  \theta_{\alpha}\left(x - x_j\right) = \sum_{k \in \mathbb{Z}} e^{-\pi \alpha k^2} \left( \sum_{j=1}^{n} e^{-2 \pi i k x_j} \right) e^{2 \pi i k x}$$
is not quite of the required form since it is not a trigonometric polynomial. However, a simple application of the triangle inequality leads to
$$ \left\| \sum_{j=1}^n  \theta_{\alpha}\left(x - x_j\right) -  \sum_{|k| \leq n} e^{-\pi \alpha k^2} \left( \sum_{j=1}^{n} e^{-2 \pi i k x_j} \right) e^{2 \pi i k x} \right\|_{L^{\infty}} \lesssim  n  e^{-\pi \alpha (n+1)^2}.$$
We apply the McGehee--Pigno--Smith inequality to the trigonometric polynomial
$$  \left\|\sum_{|k| \leq n} e^{-\pi \alpha k^2} \left( \sum_{j=1}^{n} e^{-2 \pi i k x_j} \right) e^{2 \pi i k x}  - n \right\|_{L^1} \geq \frac{1}{400n}  \sum_{\substack{k \neq 0 \\ |k| \leq n}}  e^{-\pi \alpha k^2} \left| \sum_{j=1}^{n} e^{-2 \pi i k x_j} \right|.$$
Combined with the truncation error, this leads to the lower bound
$$ \left\|  \sum_{j=1}^n  \theta_{\alpha}\left(x - x_j\right)  - n\right\|_{L^1} \geq \frac{1}{400n}  \sum_{\substack{k \neq 0 \\ |k| \leq n}}  e^{-\pi \alpha k^2} \left| \sum_{j=1}^{n} e^{-2 \pi i k x_j} \right| - \mathcal{O}(n  e^{-\pi \alpha (n+1)^2}).$$
\end{proof}

\subsection{The first \texorpdfstring{$n-1$}{n-1} Fourier coefficients are small.}\label{sec:small_Fourier_coeff}
The purpose of this section is to show that the first $n-1$ Fourier coefficients of any minimizing configuration are exponentially small in~$n$.
\begin{lemma}
    Let $\{x_1, \dots, x_n\} \subset [0,1)$ be a configuration of points that maximizes the minimum. Then
    $$ \max_{1 \leq |k| \leq n-1}  \left| \sum_{j=1}^{n} e^{-2 \pi i k x_j} \right| \leq 2000n^2 \cdot e^{-\pi \alpha (2n-1)}.$$
\end{lemma}
\begin{proof}
Combining Lemma \ref{lem:basic} with Lemma \ref{lem:McGPS_estimate}, we deduce that for any set $\left\{x_1, \dots, x_n \right\} \subset [0,1]$ the function
$$ f(x) = \sum_{k \in \mathbb{Z}} e^{-\pi \alpha k^2} \left( \sum_{j=1}^{n} e^{-2 \pi i k x_j} \right) e^{2 \pi i k x}$$
satisfies the inequality
$$ \min_{0 \leq x \leq 1} f(x) - n \leq   -\frac{1}{800n}  \sum_{\substack{k \neq 0 \\ |k| \leq n}}  e^{-\pi \alpha k^2} \left| \sum_{j=1}^{n} e^{-2 \pi i k x_j} \right| + \mathcal{O}(n  e^{- \pi \alpha (n+1)^2}).$$
We know that equispaced points satisfy
$$ \min_{0 \leq x \leq 1} \sum_{j=1}^n \theta_\alpha(x-x_j) = n - 2 n e^{-\pi \alpha n^2}  + \mathcal{O}(n e^{-4 \pi \alpha n^2}).$$
Therefore, if we now assume that $\left\{x_1, \dots, x_n \right\} \subset [0,1)$ is a configuration maximizing the minimum, we have that
$$  \min_{0 \leq x \leq 1} \sum_{j=1}^n \theta_\alpha(x-x_j) \geq   n - 2 n e^{-\pi \alpha n^2}  + \mathcal{O}(n e^{-4 \pi \alpha n^2}).$$
which then implies
$$ \frac{1}{800n}  \sum_{\substack{k \neq 0 \\ |k| \leq n}}  e^{-\pi \alpha k^2} \left| \sum_{j=1}^{n} e^{-2 \pi i k x_j} \right| \leq 2 n e^{-\pi \alpha n^2} + \mathcal{O}(n  e^{-\pi \alpha (n+1)^2}).$$
This implies that for $1 \leq |k| \leq n$ and $n$ sufficiently large (depending only on $\alpha$)
$$ \left| \sum_{j=1}^{n} e^{-2 \pi i k x_j} \right| \leq 2000 n^2 \cdot e^{-\pi \alpha (n^2 - k^2)}.$$
This allows us to conclude that the first $n-1$ Fourier coefficients of the measure given by the sum of the $n$ Dirac measures in $x_1, \dots, x_n$ is exponentially small
\begin{equation}\label{eq:Fourier_coeff_small}
    \max_{1 \leq |k| \leq n-1}  \left| \sum_{j=1}^{n} e^{-2 \pi i k x_j} \right| \leq 2000n^2 \cdot e^{-\pi \alpha (2n-1)}.
\end{equation}
\end{proof}

\textit{Remark.} We note that the proof actually shows quite a bit more since the last step of the argument is only sharp when $k=n-1$. We note the stronger inequality
$$ \left| \sum_{j=1}^{n} e^{-2 \pi i k x_j} \right| \leq 2000 n^2 \cdot e^{-\pi \alpha (n^2 - k^2)}$$
but this will not strictly be required in the remainder of the argument.

\subsection{The gaps are regular}\label{sec:regular_gaps}
If we have $n$ equispaced points, then the first $n-1$ Fourier coefficients vanish. We prove a stability version of this statement: if the first $n-1$ Fourier coefficients are small, the points are almost equispaced.
\begin{lemma}\label{lem:regular_gaps}
    Suppose $\left\{x_1, \dots, x_n \right\} \subset [0,1)$ has the property that
     $$ \max_{1 \leq |k| \leq n-1}  \left| \sum_{j=1}^{n} e^{-2 \pi i k x_j} \right| \leq \varepsilon.$$
    Then, for $\varepsilon > 0$ sufficiently small (say $\varepsilon \leq 1/(1000n^4)$), there exists a permutation $\pi:S_n \rightarrow S_n$ and a global shift $z \in [0,1]$ such that
     $$ \left| x_{\pi(j)} - \frac{j}{n} - z \right| \leq \varepsilon.$$
\end{lemma}
\begin{proof}
We use the Fej\'er kernel
$$ F_n(x) = \sum_{|k| \leq n} \left(1 - \frac{|k|}{n}\right) e^{2\pi i k x} = \frac{1}{n} \left(\frac{\sin{(\pi n x)}}{\sin{( \pi x)}} \right)^2 \geq 0.$$
Note that  $F_n(0) = n$. Therefore
\begin{align*}
    \sum_{i,j=1}^{n} F_n(x_i - x_j)
    & = \sum_{i,j=1}^{n} \sum_{|k| \leq n} \left(1 - \frac{|k|}{n}\right) e^{2\pi i k (x_i- x_j)} \\
    & = \sum_{|k| \leq n} \left(1 - \frac{|k|}{n}\right) \sum_{i,j=1}^{n}  e^{2\pi i k (x_i- x_j)} \\
    & = \sum_{|k| \leq n} \left(1 - \frac{|k|}{n}\right) \left| \sum_{j=1}^{n} e^{2\pi i k x_j} \right|^2.
\end{align*}

Hence, applying the assumption of the first $n-1$ non-zero Fourier coefficients being small, we get
\begin{align*}
    \sum_{i,j=1}^{n} F_n(x_i - x_j) 
    &= n^2 +  \sum_{\substack{i, j=1 \\ i \neq j}}^{n} F_n(x_i - x_j) = \sum_{|k| \leq n} \left(1 - \frac{|k|}{n}\right) \left| \sum_{j=1}^{n} e^{2\pi i k x_j} \right|^2 \\
    &= n^2 +  \sum_{\substack{ |k| \leq n \\ k \neq 0}} \left(1 - \frac{|k|}{n}\right) \left| \sum_{j=1}^{n} e^{2\pi i k x_j} \right|^2 \leq n^2 + 2n \varepsilon^2.
\end{align*}
From the above calculation we also conclude that, for any index $i \neq j$,
$$ F_n(x_i - x_j) \leq  \sum_{\substack{i, j=1 \\ i \neq j}}^{n} F_n(x_i - x_j) \leq 2n \varepsilon^2.$$
This inequality, by itself, is not tremendously powerful: we bound a term by a sum containing $\sim n^2$ similar terms. However, we have the luxury that we will only apply the Lemma in a regime where $\varepsilon$ is already exponentially small in $n$ which allows for losses at a polynomial scale. The roots of $F_n$ on $[0,1)$ are exactly the points of the form $k/n$ for $1 \leq k \leq n-1$. 
Since
$$ F_n(x) = \frac{1}{n} \left(\frac{\sin{(\pi n x)}}{\sin{( \pi x)}} \right)^2$$
we have
$$ \frac{d^2}{dx^2} F_n(x) = \frac{2 \pi ^2 \csc ^2(\pi  x)}{n} \, X$$
where
$$ X =\left(n^2 \cos ^2(\pi  n x)-\sin ^2(\pi  n x) \left(n^2-3
   \csc ^2(\pi  x)+2\right)-2 n \cot (\pi  x) \sin (2 \pi  n x)\right).$$
   At points of the form $x = k/n$ this expression simplifies to
$$ \frac{d^2}{dx^2} F_n(x) \Big|_{x = \frac{k}{n}} = 2 \pi ^2 n \csc ^2\left(\frac{k \pi }{n}\right) \geq 2 \pi^2 n.$$
Therefore, for $y$ sufficiently close to 0, we have
$$ F_n\left( \frac{k}{n} + y\right) \geq 2n y^2.$$
A similar argument can be used to give an upper bound on the third derivative. The Taylor formula with remainder shows that the inequality is valid for $y$ in a region around 0 that shrinks polynomially in $n$ and from this we deduce the validity of the inequality for $\varepsilon$ sufficiently small. The previous inequality
$$ F_n(x_i - x_j) \leq  \sum_{\substack{i, j=1 \\ i \neq j}}^{n} F_n(x_i - x_j) \leq 2n \varepsilon^2$$
implies that $x_i - x_j$ has to be of the form $x_i - x_j = k/n + \delta$ with some $\delta \leq \varepsilon$. Moreover, since $F_n(0) = n$ we can also deduce that $|x_i - x_j| > 1/2n$ (provided $\varepsilon$ is sufficiently small) which then forces the existence of a global perturbation.
\end{proof}

\section*{Part 2 of the proof}
\subsection{The Main Contribution.}\label{sec:main_contribution}
We quickly recall what we already know from the first part of the proof. We know that any optimal configuration $\left\{x_1, \dots, x_n\right\}$ has to be close to the case of equispaced points. More precisely, it has to be of the form
$$ x_i = \frac{i}{n} + z + \varepsilon_i \qquad \mbox{where} \qquad |\varepsilon_i| \leq 2000n^2 \cdot e^{-\pi \alpha (2n-1)}$$
is exponentially small in $n$ and $z \in \mathbb{R}$ is an arbitrary shift. By translation symmetry, we can assume that $z=0$ and $\sum_j \varepsilon_j = 0$ and will do so in all subsequent arguments.

We can rewrite the sum over $\theta-$functions as a Fourier series
\begin{align*}
     \sum_{j=1}^{n} \theta_\alpha \left(x - \frac{j}{n} + \varepsilon_j \right) &=  \sum_{j=1}^{n}  \sum_{k \in \mathbb{Z}} e^{-\pi \alpha k^2} e^{2 \pi i k (x - \frac{j}{n} + \varepsilon_j )} \\
     &= \sum_{k \in \mathbb{Z}} e^{-\pi \alpha k^2} \left( \sum_{j=1}^{n} e^{2 \pi i k (\varepsilon_j - \frac{j}{n}) }\right)e^{2 \pi i k x} \\
     &=  \sum_{k \in \mathbb{Z}} e^{-\pi \alpha k^2} \left( \sum_{j=1}^{n}  e^{2 \pi i k \varepsilon_j } e^{-2 \pi i k  \frac{j}{n} }\right)e^{2 \pi i k x}.
\end{align*}
We remark that, as already noted above, when all the $\varepsilon_j = 0$, then
$$
    \sum_{j=1}^{n} \theta_\alpha \left(x - \frac{j}{n} \right) = n  + 2 n e^{-\pi \alpha n^2} \cos{(2 \pi n x)} + \mathcal{O}(n e^{-4 \pi \alpha n^2}).
$$
In that case, the minimal value is very close to the mean value $n$. It remains to show that small perturbations decrease the minimal value. Using the Taylor formula with the remainder term we note that the frequency $k =n$ contributes
\begin{align*}
    e^{-\pi \alpha n^2} \left( \sum_{j=1}^{n}  e^{2 \pi i n \varepsilon_j } \right)e^{2 \pi i n x}
    &=  e^{-\pi \alpha n^2} \left( 
        n+ 2 \pi i n\sum_{j=1}^{n}  \varepsilon_j  + \mathcal{O}\left(n^2  \sum_{j=1}^{n} \varepsilon_j^2\right) \right)e^{2 \pi i n x} \\
    &=e^{-\pi \alpha n^2} n e^{2 \pi i n x} + \mathcal{O}\left(n^2  e^{-\pi \alpha n^2}  \sum_{j=1}^{n} \varepsilon_j^2\right)
 \end{align*}
and the same contribution arises for $k =-n$. Thus the three terms
$$ B = \sum_{k \in \left\{-n, 0, n\right\}} e^{-\pi \alpha k^2} \left( \sum_{j=1}^{n}  e^{2 \pi i k \varepsilon_j } e^{-2 \pi i k  \frac{j}{n} }\right)e^{2 \pi i k x} $$ 
contribute, up to a small error term, the same quantity as the unperturbed case $\varepsilon_j = 0$ and 
$$ B = n  + 2 n e^{-\pi \alpha n^2} \cos{(2 \pi n x)} +\mathcal{O}\left(n^2  e^{-\pi \alpha n^2}  \sum_{j=1}^{n} \varepsilon_j^2\right).$$ 
Recall that, in the unperturbed case, the minima are attained at \mbox{$(k+1/2)/n$}, $0 \leq k \leq n-1$. We will show that a small perturbation necessarily makes one of the minima smaller and argue by contradiction: if there was a small perturbation of the points that increases the minimum, then, in particular, the size of the perturbation would have to be positive at all points of the form \mbox{$(k+1/2)/n$}, $0 \leq k \leq n-1$ (since that is where the minima are attained in the unperturbed case). The remainder of the argument is dedicated to showing that this cannot happen.

\subsection{A Trigonometric Lemma.}\label{sec:trig_poly}
This section proves a self-contained Lemma, which shows that a trigonometric polynomial of degree at most $(n-1)/2$ assumes negative values at at least one of the points $(k+1/2)/n$, for $0 \leq k \leq n-1$. The obtained bound is likely far from optimal but suffices for our purpose. Indeed, the rapid decay of the Gaussian weight ensures that any type of polynomial bound would suffice for the remainder of the argument.

\begin{lemma} \label{lem:small}
    If $f:[0,1] \rightarrow \mathbb{R}$ is a trigonometric polynomial of the form
    $$ f(x) = \sum_{1 \leq |j| \leq \frac{n-1}{2}} a_j e^{2 \pi i j x} \quad \mbox{then} \quad \min_{0 \leq k \leq n-1} f\left(\frac{k+1/2}{n}\right) \leq -\frac{\|f\|_{L^2([0,1])} }{3 n^{2}}.$$
\end{lemma}
We note that the restriction on the frequency $|j| \leq (n-1)/2$ is tight. Suppose $n$ is even and consider the trigonometric polynomial
$$ f(x) =  e^{2\pi i \frac{n}{2} x} +  e^{-2\pi i \frac{n}{2} x} = 2\cos{(\pi n x)}$$
which satisfies
$$ f\left(\frac{k+1/2}{n}\right) = 2\cos{(\pi (k+1/2))} = 0.$$

Before stating the proof of Lemma \ref{lem:small}, we establish one of the two main ingredients: a modified Poincar\'{e} inequality for functions that do not quite vanish on the boundary. Needless to say, the tools and arguments used to establish this inequality are completely standard and we do not claim the inequality to be novel in any sense. Many similar inequalities are known in the general context of trace inequalities and embedding results for Sobolev spaces.
\begin{proposition}[Modified Poincar\'{e} Inequality] Let $f:[a,b] \rightarrow \mathbb{R}$ be continuous and differentiable on $(a,b)$ satisfying $|f(a)| \leq M$ and
$|f(b)| \leq M$. Then
\begin{align*}
    \int_a^b f(x)^2 dx &\leq M^2 (b-a) + 2M \sqrt{b-a} \left(\frac{(b-a)^2}{\pi^2} \int_a^b f'(x)^2  dx \right)^{1/2} \\
    &+ \frac{(b-a)^2}{\pi^2} \int_a^b f'(x)^2  dx
\end{align*}
\end{proposition}
\begin{proof}
The following makes sense in the more general Sobolev space $H^1$ (as opposed to the smaller space $C^1$) but this will not be relevant here.
We first note that replacing $f(x)$ by $|f(x)|$ does not change $\|f\|_{L^2}$ and leaves $\|f'\|_{L^2}$ invariant. It thus suffices to prove the inequality for non-negative $f(x)$. We proceed with basic facts: the first is the standard Poincar\'{e} inequality, implying that if $g:[c,d] \rightarrow \mathbb{R}$
satisfies $g(c) = g(d) = 0$ then 
$$ \int_c^d g(x)^2 dx \leq \frac{(d-c)^2}{\pi^2} \int_c^d g'(x)^2 dx.$$
This one-dimensional inequality is sometimes known as the \textit{Wirtinger inequality} (for example in Blaschke's 1916 book \textit{Kreis und Kugel} \cite{blaschke}). However, we note that it seems to have been discovered many times: for example, Hurwitz \cite{hurwitz} already used it in his 1901 proof of the isoperimetric inequality. We refer to Payne and Weinberger \cite{PayWei60} or work of the second author \cite{Stei2} for more on Poincar\'{e} inequalities.
This inequality then implies that 
$$ \int_a^b (f(x)-M)^2 \mathbf{1}_{f(x) \geq M} \, dx \leq \frac{(b-a)^2}{\pi^2} \int_a^b f'(x)^2 \mathbf{1}_{f(x) \geq M} \, dx$$
which we can square out and write as
\begin{align*}
    \int_a^b f(x)^2\mathbf{1}_{f(x) \geq M} \, dx + M^2 \, |\left\{ f \geq M \right\}|
    &\leq 2M \int_a^b f(x)\mathbf{1}_{f(x) \geq M} \, dx\\
    & \quad + \frac{(b-a)^2}{\pi^2} \int_a^b f'(x)^2 \mathbf{1}_{f(x) \geq M} \, dx
 \end{align*}
The first integral on the right-hand side can be bounded with Cauchy--Schwarz
\begin{align*}
    \int_a^b f(x)\mathbf{1}_{f(x) \geq M} \, dx \leq |\left\{f \geq M\right\}|^{1/2} \left( \int_a^b f(x)^2\mathbf{1}_{f(x) \geq M} \, dx \right)^{1/2}
\end{align*}
which leads to the estimate, abbreviating $Z = \left( \int_a^b f(x)^2\mathbf{1}_{f(x) \geq M} \, dx \right)^{1/2}$,
$$ Z^2 - 2 M \, |\left\{f \geq M\right\}|^{1/2} Z + M^2 \, |\left\{f \geq M\right\}| \leq  \frac{(b-a)^2}{\pi^2} \int_a^b f'(x)^2 \mathbf{1}_{f(x) \geq M} \, dx.$$
The left-hand side can be factored as
$$ (Z - M \, |\left\{f \geq M\right\}|^{1/2})^2 \leq  \frac{(b-a)^2}{\pi^2} \int_a^b f'(x)^2 \mathbf{1}_{f(x) \geq M} \, dx$$
and thus
$$ Z \leq M \, |\left\{f \geq M\right\}|^{1/2} +  \left( \frac{(b-a)^2}{\pi^2} \int_a^b f'(x)^2 \mathbf{1}_{f(x) \geq M} \, dx\right)^{1/2}$$
We also have the trivial estimate
$$ \int_a^b f(x)^2 \mathbf{1}_{f(x) \leq M} \, dx \leq \, |\left\{ f(x) \leq M \right\}| \, M^2.$$
Adding the last estimate to the square of the penultimate estimate and using $|\left\{f \geq M\right\}| \leq b-a$, we arrive at
\begin{align*}
    \int_a^b f(x)^2 dx &\leq M^2 (b-a) + 2 M (b-a)^{1/2}  \left( \frac{(b-a)^2}{\pi^2} \int_a^b f'(x)^2 \mathbf{1}_{f(x) \geq M} \, dx\right)^{1/2}\\
    & \quad + \frac{(b-a)^2}{\pi^2} \int_a^b f'(x)^2 \mathbf{1}_{f(x) \geq M} \, dx.
\end{align*}
\end{proof}

\begin{proof}[Proof of Lemma \ref{lem:small}]
   The minimum is necessarily $\leq 0$ since
    \begin{align*}
         \min_{0 \leq k \leq n-1} f\left(\frac{k+1/2}{n}\right) &\leq \frac{1}{n} \sum_{0 \leq k \leq n-1} f\left(\frac{k+1/2}{n}\right) \\
         &= \frac{1}{n} \sum_{k = 0}^{n-1}   \sum_{1 \leq |j| \leq n-1} a_j e^{2 \pi i j (k+1/2)/n} \\
&= \frac{1}{n} \sum_{1 \leq |j| \leq n-1}  \sum_{k = 0}^{n-1}   a_j e^{2 \pi i \frac{j}{2n}} e^{2 \pi i j k/n} \\
&=  \frac{1}{n} \sum_{1 \leq |j| \leq n-1}    a_j e^{2 \pi i \frac{j}{2n}} \sum_{k = 0}^{n-1}  e^{2 \pi i j k/n} = 0.
    \end{align*}
Let us now assume that the minimum is negative but very close to 0
    $$ X = \min_{0 \leq k \leq n-1} f\left(\frac{k+1/2}{n}\right) \leq 0.$$
Roots of unity summing to 0 then shows, just as above, that
\begin{align}
   0 &=  \sum_{0 \leq k \leq n-1} f\left(\frac{k+1/2}{n}\right) \geq (n-1) X + \max_{0 \leq k \leq n-1} f\left(\frac{k+1/2}{n}\right)\\
   &\geq n X + \max_{0 \leq k \leq n-1} f\left(\frac{k+1/2}{n}\right).
\end{align}
from which we deduce
$$ \max_{0 \leq k \leq n-1} f\left(\frac{k+1/2}{n}\right)  \leq (-X)n.$$
Using this in combination with the modified Poincar\'{e} inequality with $M = n |X|$ we deduce 
\begin{align*}
\int_{\frac{k+1/2}{n}}^{\frac{k+3/2}{n}} f(x)^2 dx &\leq nX^2 +  2|X|\sqrt{n} \sqrt{\frac{1}{n^2 \pi^2} \int_{\frac{k+1/2}{n}}^{\frac{k+3/2}{n}} f'(x)^2 dx} \\
&+\frac{1}{n^2 \pi^2} \int_{\frac{k+1/2}{n}}^{\frac{k+3/2}{n}} f'(x)^2 dx.
\end{align*}
Summing over all the intervals (periodically interpreted), we get
\begin{align*}
 \int_0^1 f(x)^2 \, dx &\leq n^2 X^2 +  \frac{1}{n^2 \pi^2}\int_0^1 f'(x)^2 \, dx \\
 &+  2|X|\sqrt{n} \sum_{k=0}^{n-1}   \sqrt{\frac{1}{n^2 \pi^2} \int_{\frac{k+1/2}{n}}^{\frac{k+3/2}{n}} f'(x)^2 \, dx}
\end{align*}
As for the remaining sum, we use the Cauchy--Schwarz inequality to bound
\begin{align*}
     \sum_{k=0}^{n-1} 1 \cdot  \sqrt{\frac{1}{n^2 \pi^2} \int_{\frac{k+1/2}{n}}^{\frac{k+3/2}{n}} f'(x)^2 \, dx} &\leq \sqrt{n} \sqrt{ \sum_{k=0}^{n-1} \frac{1}{n^2 \pi^2} \int_{\frac{k+1/2}{n}}^{\frac{k+3/2}{n}} f'(x)^2 \, dx} \\
     &= \sqrt{n}  \sqrt{  \frac{1}{n^2 \pi^2} \int_{0}^{1} f'(x)^2 \, dx}
\end{align*}
Altogether, this implies
\begin{align*}
    \int_0^1 f(x)^2 \, dx &\leq n^2 X^2 +  \frac{1}{n^2 \pi^2}\int_0^1 f'(x)^2 \, dx + 2|X|n \sqrt{  \frac{1}{n^2 \pi^2} \int_{0}^{1} f'(x)^2 \, dx}.
\end{align*}

As $f$ is a trigonometric polynomial of degree at most $(n-1)/2$, we have
$$ \int_0^1 f(x)^2 dx =  \sum_{1 \leq |j| \leq \frac{n-1}{2}} |a_j|^2$$
    as well as
    \begin{align*}
     \int_0^1 f'(x)^2 dx &=  \sum_{1 \leq |j| \leq \frac{n-1}{2}} (2 \pi j)^2 |a_j|^2 \\
     &\leq  (n-1)^2 \pi^2  \sum_{1 \leq |j| \leq \frac{n-1}{2}}  |a_j|^2 = (n-1)^2 \pi^2  \int_0^1 f(x)^2 \,dx.
     \end{align*}
Plugging this in, we get
\begin{align*}
       \int_0^1 f(x)^2 dx  &\leq \frac{1}{n^2 \pi^2}  \int_{0}^1 f'(x)^2 \, dx +X^2 n^2  + 2|X| \sqrt{  \frac{1}{n^2 \pi^2} \int_{0}^{1} f'(x)^2 \, dx}  \\
       &\leq \left(1 - \frac{1}{n}\right)^2 \int_0^1 f(x)^2 \, dx  + X^2 n^2 + 2|X| n \left(1 - \frac{1}{n}\right) \sqrt{\int_0^1 f(x)^2 \, dx}.
\end{align*}
For an arbitrary parameter $\gamma > 0$, the inequality
$$ \gamma^2 \leq \left(1- \frac{1}{n}\right)^2 \gamma^2 + X^2 n^2 + 2|X| n \left(1 - \frac{1}{n}\right) \gamma$$
can be simplified using $(1-1/n)^2 \leq 1-1/n$ and $(1-1/n) \leq 1$ to imply that
$$   X^2 n^2 + 2|X| n \gamma - \frac{\gamma^2}{n} \geq 0$$
which, by solving the quadratic equation can be seen to imply that
$$ |X| \geq \frac{\sqrt{n^2+n} - n}{n^2}\gamma \geq   \frac{\gamma}{3 n^2}$$
which is the desired result.
\end{proof}

\begin{remark}
    Much of the difficulty comes from the fact that we only evaluate the trigonometric polynomial in equispaced points. If one was just interested in the minimum being small in \textit{some} place, there is a very elementary argument which we conclude for the sake of context.
\end{remark}

\begin{lemma}
    Let $f(x) = \sum_{1 \leq |j| \leq n-1} a_j e^{2\pi i j x}$ be a real-valued trigonometric polynomial. Then
    $$\min_{0 \leq x \leq 1} f(x) \leq - \frac{\|f\|_{L^2}}{3 \sqrt{n}}.$$
\end{lemma}
\begin{proof}
    We also have the trivial estimate
    $$ \int_0^1 f(x)^2 \leq \|f\|_{L^{\infty}} \|f\|_{L^1} \quad \mbox{and thus} \quad \|f\|_{L^1} \geq \frac{ \|f\|_{L^2}^2}{\|f\|_{L^{\infty}}}.$$
    Appealing to Lemma \ref{lem:basic}, we deduce
      $$  \min_{0 \leq x \leq 1} f(x)  \leq - \frac{1}{2} \|f\|_{L^1} \leq - \frac{1}{2} \frac{ \|f\|_{L^2}^2}{\|f\|_{L^{\infty}}}.$$
  We have, via Plancherel, that
  $$ \|f\|_{L^2}^2 = \sum_{1 \leq j \leq n-1} |a_j|^2$$
  and, via the triangle inequality and Cauchy--Schwarz inequality, that
    $$ \|f\|_{L^{\infty}} \leq \sum_{1 \leq j \leq n-1} |a_j| \leq \sqrt{2n} \left( \sum_{1 \leq j \leq n-1} |a_j|^2\right)^{1/2} = \sqrt{2n} \|f\|_{L^2}.$$
    From this and $2 \sqrt{2} \leq 3$ the result follows.
\end{proof}

\subsection{Outline of the remaining argument.}\label{sec:outline_g2_h}
In this section we outline how the argument will be concluded. We first recall that
\begin{align*}
     \sum_{j=1}^{n} \theta_\alpha \left(x - \frac{j}{n} + \varepsilon_j \right) = \sum_{k \in \mathbb{Z}} e^{-\pi \alpha k^2} \left( \sum_{j=1}^{n}  e^{2 \pi i k \varepsilon_j } e^{-2 \pi i k  \frac{j}{n} }\right)e^{2 \pi i k x}.
\end{align*}
The main contribution is coming from the three terms $k \in \{-n,0,n\}$
$$ B = \sum_{k \in \left\{-n, 0, n\right\}} e^{-\pi \alpha k^2} \left( \sum_{j=1}^{n}  e^{2 \pi i k \varepsilon_j } e^{-2 \pi i k  \frac{j}{n} }\right)e^{2 \pi i k x} $$ 
which contribute
$$ B = n  + 2 n e^{-\pi \alpha n^2} \cos{(2 \pi n x)} +\mathcal{O}\left( n^2  e^{-\pi \alpha n^2}  \sum_{j=1}^{n} \varepsilon_j^2\right).$$
We will choose to sum over even more terms (even though they are rather small), namely $k \in n \Z$, so as to allow for a comparison to the minimal value attained by equidistributed points. For this purpose we set
\begin{align*}
    A(x) &= \sum_{k \in n \mathbb{Z}} e^{-\pi \alpha k^2} \left( \sum_{j=1}^{n}  e^{2 \pi i k \varepsilon_j } e^{-2 \pi i k  \frac{j}{n} }\right)e^{2 \pi i k x} \\
    &=  \sum_{k \in n \mathbb{Z}} e^{-\pi \alpha k^2} \left( \sum_{j=1}^{n}  e^{2 \pi i k \varepsilon_j } \right)e^{2 \pi i k x}, 
\end{align*}
where the simplification comes from the fact that these exponential expressions are all 1 when $k$ is a multiple of $n$.
In particular, all the Fourier coefficients are reasonably close to $n$. More precisely, using again that the sum over all displacements $\varepsilon_j$ equals 0, we get
\begin{align*}
 \sum_{j=1}^{n}  e^{2 \pi i k \varepsilon_j} =   n +  \sum_{j=1}^{n}  \left( e^{2 \pi i k \varepsilon_j } - 1 \right)= n +  \sum_{j=1}^{n}  \left( e^{2 \pi i k \varepsilon_j } - 1 - 2 \pi i k \varepsilon_j \right).
\end{align*}
We have, for all $x \in \mathbb{R}$ that
$| e^{ix} - 1 - ix| \leq x^2$
and thus
$$ \left| -n + \sum_{j=1}^{n}  e^{2 \pi i k \varepsilon_j } \right| \leq 4 \pi^2 k^2 \sum_{j=1}^{n}  \varepsilon_j^2.$$
Combining this with
$$ \sum_{\substack{k \in n \mathbb{Z} \\ |k| >n }} k^2 e^{-\pi \alpha k^2} \lesssim n^2 e^{-\pi \alpha n^2} $$
we deduce that
$$ A(x)  =   \sum_{j=1}^{n} \theta_\alpha\left(x - \frac{j}{n}  \right) +\mathcal{O}\left( n^2 e^{-\pi \alpha n^2}  \sum_{j=1}^{n} \varepsilon_j^2\right).$$
It is our goal to show that the perturbation induced by $\varepsilon_j \neq 0$ has to decrease the value in at least one of the minima. To this end, we split the function as
$$  \sum_{j=1}^{n} \theta_\alpha \left(x - \frac{j}{n} + \varepsilon_j \right)  = A(x) + g_1(x) + g_2(x) + h(x), $$
where $A$ sums over all multiples of $n$, $g_1$ sums over the first $(n-1)/2$ frequencies, $g_2$ sums frequencies between $(n-1)/2$ and $n-1$ and $h$ sums over the rest, frequencies larger than $n$ and where $n$ does not divide $k$. Thus
\begin{align*}
    g_1(x)  &=    \sum_{1 \leq |k| \leq \frac{n-1}{2}} e^{-\pi \alpha k^2} \left( \sum_{j=1}^{n}  e^{2 \pi i k \varepsilon_j } e^{-2 \pi i k  \frac{j}{n} }\right)e^{2 \pi i k x}
\end{align*}
while $g_2$ sums over the remaining small frequencies
\begin{align*}
    g_2(x)  &=    \sum_{\frac{n-1}{2} <  |k| \leq n-1} e^{-\pi \alpha k^2} \left( \sum_{j=1}^{n}  e^{2 \pi i k \varepsilon_j } e^{-2 \pi i k  \frac{j}{n} }\right)e^{2 \pi i k x}, 
    \end{align*}
    and $h$ sums over the remaining terms
 \begin{align*}   
    h(x)    &=    \sum_{\substack{ |k| \geq n+1 \\ n \, \nmid \, k}} e^{-\pi \alpha k^2} \left( \sum_{j=1}^{n}  e^{2 \pi i k \varepsilon_j } e^{-2 \pi i k  \frac{j}{n} }\right)e^{2 \pi i k x}.
\end{align*}
The remaining argument proceeds as follows
\begin{enumerate}
    \item We show, in the next section, that $\|g_1\|_{L^2}$ is not too small (in terms of $\sum_{j=1}^{n} \varepsilon_j^2$). The Discrete Fourier Transform naturally arises in the process.
    \item Lemma \ref{lem:small} then implies that
     $$\min_{0 \leq k \leq n-1} g_1\left(\frac{k+1/2}{n}\right) \leq - \frac{\|g_1\|_{L^2}}{3n^{2}} $$
     is fairly negative.
     \item We show $\|g_2\|_{L^{\infty}} \ll \|g_1\|_{L^2}$ (which follows again from the properties of the Discrete Fourier Transform) and that the same is true for $h$. 
     \item Thus the sum of the three terms is fairly negative in at least one of the points of the form $(k+1/2)/n$ and this then implies the result.
\end{enumerate}

\subsection{Discrete Fourier Transform.}\label{sec:DFT}
We recall again that
\begin{align*}
     \sum_{j=1}^{n} \theta_\alpha \left(x - \frac{j}{n} + \varepsilon_j \right) =  \sum_{k \in \mathbb{Z}} e^{-\pi \alpha k^2} \left( \sum_{j=1}^{n}  e^{2 \pi i k \varepsilon_j } e^{-2 \pi i k  \frac{j}{n} }\right)e^{2 \pi i k x}.
\end{align*}
We also note that the $\varepsilon_j$ are fairly small: \eqref{eq:Fourier_coeff_small} together with the proof of Lemma \ref{lem:regular_gaps} gives
$$ \max_{1 \leq j \leq n} |\varepsilon_j| = \mathcal{O} \left(  n^2 \cdot e^{- 2 \pi \alpha n} \right),$$
where the implicit constant depends on $\alpha$.
As it turns out, since these are exponentially small in $n$, the basic Taylor expansion
$$ e^{2 \pi i k \varepsilon_j } = 1 + 2 \pi i k \varepsilon_j + \mathcal{O}(k^2 \varepsilon_j^2)$$ 
is highly accurate and we deduce, as long as $k$ is not a multiple of $n$, that
\begin{equation}
    \sum_{j=1}^{n}  e^{2 \pi i k \varepsilon_j } e^{-2 \pi i k  \frac{j}{n} } = 2 \pi i k  \sum_{j=1}^{n}  \varepsilon_j  e^{-2 \pi i k  \frac{j}{n} } + \mathcal{O}\left(k^2 \sum_{j=1}^{n} \varepsilon_j^2\right).
\end{equation}
We observe that this is, up to various types of rescaling, simply a Discrete Fourier Transform of $(\varepsilon_1, \dots, \varepsilon_n)$. Since the $\varepsilon_j$ are all real-valued, we have the symmetry
$$ \left| \sum_{j=1}^{n}  \varepsilon_j  e^{-2 \pi i k  \frac{(n-j)}{n} } \right| =  \left| \sum_{j=1}^{n}  \varepsilon_j  e^{-2 \pi i k  \frac{j}{n} } \right| $$
The Discrete Fourier Transform preserves the $\ell^2$-norm and therefore
$$ n\sum_{j=1}^{n} \varepsilon_j^2 = \sum_{k=1}^{n-1} \left| \sum_{j=1}^{n}  \varepsilon_j  e^{-2 \pi i k  \frac{j}{n} } \right|^2$$
where we omit the $k=0$ term because $\varepsilon_1 + \dots + \varepsilon_n = 0$.
This implies
\begin{align*}
     \sum_{j=1}^{n} \varepsilon_j^2 &= \frac{1}{n}  \sum_{k=1}^{n-1} \left| \sum_{j=1}^{n}  \varepsilon_j  e^{-2 \pi i k  \frac{j}{n} } \right|^2 
     \leq  \max_{1 \leq k \leq n-1} \left| \sum_{j=1}^{n}  \varepsilon_j  e^{-2 \pi i k  \frac{j}{n} } \right|^2.
\end{align*}
This immediately implies that at least one Fourier coefficient is large
$$ \max_{1 \leq k \leq n-1} \left| \sum_{j=1}^{n}  \varepsilon_j  e^{-2 \pi i k  \frac{j}{n} } \right| \geq \left( \sum_{j=1}^{n} \varepsilon_j^2 \right)^{1/2} \gg  \mathcal{O}\left(n^2 \sum_{j=1}^{n} \varepsilon_j^2\right)$$
and, in particular, is many orders of magnitude larger than the error terms (recall that the error terms are exponentially small in $n$).

\subsection{The final estimates.}\label{sec:final_estimate}
This also implies, using the Plancherel identity, that $g_1$ is large in $L^2$ since
\begin{align*}
    \|g_1(x)\|_{L^2} &=   \left\| \sum_{1 \leq |k| \leq \frac{n-1}{2}} e^{-\pi \alpha k^2} \left( \sum_{j=1}^{n}  e^{2 \pi i k \varepsilon_j } e^{-2 \pi i k  \frac{j}{n} }\right)e^{2 \pi i k x} \right\|_{L^2} \\
    &= \left(  \sum_{1 \leq |k| \leq \frac{n-1}{2}} e^{-2 \pi \alpha k^2} \left| \sum_{j=1}^{n}  e^{2 \pi i k \varepsilon_j } e^{-2 \pi i k  \frac{j}{n} }\right|^2 \right)^{1/2}.
\end{align*}
The worst case is when most of the Fourier energy is localized at high frequencies and thus we can remove the smallest weight and deduce
\begin{align*}
    \|g_1(x)\|_{L^2}  \geq e^{-\pi \alpha \left( \frac{n-1}{2} \right)^2}\left(  \sum_{1 \leq |k| \leq \frac{n-1}{2}} \left| \sum_{j=1}^{n}  e^{2 \pi i k \varepsilon_j } e^{-2 \pi i k  \frac{j}{n}} \right|^2 \right)^{1/2}.
\end{align*}
At this point, we can invoke a Taylor expansion and argue that
\begin{align*}
    \|g_1(x)\|_{L^2} 
      &\geq e^{-\pi \alpha \left( \frac{n-1}{2} \right)^2}\left(  \sum_{1 \leq |k| \leq \frac{n-1}{2}} \left|  \ \sum_{j=1}^{n}  \varepsilon_j  e^{-2 \pi i k  \frac{j}{n} } +  \mathcal{O}\left(k^2 \sum_{j=1}^{n} \varepsilon_j^2\right) \right|^2 \right)^{1/2}.
\end{align*}
Now, the argument from the previous section comes into play: we do not have information about any individual
Fourier coefficient but we know that at least one of them is large 
$$ \max_{1 \leq k \leq n-1} \left| \sum_{j=1}^{n}  \varepsilon_j  e^{-2 \pi i k  \frac{j}{n} } \right| \geq \left( \sum_{j=1}^{n} \varepsilon_j^2 \right)^{1/2} \gg  \mathcal{O}\left(n^2 \sum_{j=1}^{n} \varepsilon_j^2\right)$$
and thus, for $n$ sufficiently large,
\begin{align*}
    \|g_1(x)\|_{L^2} \geq  \frac{ e^{-\pi \alpha \left( \frac{n-1}{2} \right)^2}}{2} \left(\sum_{j=1}^{n} \varepsilon_j^2 \right)^{1/2}.
\end{align*}
 It is rather easy to show that $g_2$ is many orders of magnitude smaller than $g_1$ as the Fourier coefficients are very nearly the same. Since the discrete Fourier transform has the symmetry
$$ \left| \sum_{j=1}^{n}  \varepsilon_j  e^{-2 \pi i k  \frac{(n-j)}{n} } \right| =  \left| \sum_{j=1}^{n}  \varepsilon_j  e^{-2 \pi i k  \frac{j}{n} } \right|,$$
the same Plancherel argument shows that, for $n$ sufficiently large,
$$ \| g_2\|_{L^2} \leq 20 n^{3/2}  e^{-\pi \alpha (n/2)^2}  \left(\sum_{j=1}^{n} \varepsilon_j^2 \right)^{1/2}.$$
This is exponentially smaller than $g_1(x)$ because
$$e^{-\pi \alpha (n/2)^2} \qquad \mbox{is exponentially smaller than} \qquad  e^{-\pi \alpha \left( \frac{n-1}{2} \right)^2}.$$

We will require pointwise estimates for what follows. However, the decay is sufficiently strong so that we can simply take a triangle inequality.
Using again the cancellation of the sum of roots of unity together with the fact that for $k \leq n$ we have $k^2 \varepsilon_j^2 \ll |k \varepsilon_j|$, we get for sufficiently large $n$
\begin{align*}
 \left| \sum_{j=1}^{n}  e^{2 \pi i k \varepsilon_j } e^{-2 \pi i k  \frac{j}{n} }\right| &= 
 \left| \sum_{j=1}^{n} (1 + 2\pi i k \varepsilon_j + \mathcal{O}(k^2 \varepsilon_j^2))  e^{-2 \pi i k  \frac{j}{n} }\right|\\
 &=  (1+o(1)) \left|  2 \pi i k  \sum_{j=1}^{n}   \varepsilon_j    e^{-2 \pi i k  \frac{j}{n} } \right| \leq 10k \left|\sum_{j=1}^{n}   \varepsilon_j    e^{-2 \pi i k  \frac{j}{n} } \right|.
\end{align*}
We deduce, since $k > (n-1)/2$ and thus $k \geq n/2$, that for $n$ sufficiently large,
\begin{align*}
    \left\|  g_2\right\|_{L^{\infty}} & \leq \sum_{\frac{n-1}{2} <  |k| \leq n-1} e^{-\pi \alpha k^2} \left| \sum_{j=1}^{n}  e^{2 \pi i k \varepsilon_j } e^{-2 \pi i k  \frac{j}{n} }\right| \\
    &\leq 10 \sum_{\frac{n-1}{2} <  |k| \leq n-1} e^{-\pi \alpha k^2} \, k \sum_{j=1}^{n} |\varepsilon_j|\\
    & \leq 20 n  e^{- \pi \alpha (n/2)^2}\sum_{j=1}^{n} |\varepsilon_j|  \leq 20 n^{3/2}  e^{-\pi \alpha (n/2)^2}  \left(\sum_{j=1}^{n} \varepsilon_j^2 \right)^{1/2} \ll \frac{\| g_1\|_{L^2}}{n^{100}}.
\end{align*}
A similar argument can be applied to $h$. We argue that
\begin{align*}
      \| h(x) \|_{L^{\infty}}   &=  \left\|   \sum_{\substack{ |k| \geq n+1 \\ n \, \nmid \, k}} e^{-\pi \alpha k^2} \left( \sum_{j=1}^{n}  e^{2 \pi i k \varepsilon_j } e^{-2 \pi i k  \frac{j}{n} }\right)e^{2 \pi i k x} \right\|_{L^{\infty}} \\
      &\leq    \sum_{\substack{ |k| \geq n+1 \\ n \, \nmid \, k}} e^{-\pi \alpha k^2}  2\pi k \sum_{j=1}^{n}  |\varepsilon_j|  \leq \left(\sum_{j=1}^{n} \varepsilon_j^2 \right)^{1/2} \sqrt{4 \pi^2 n}   \sum_{\substack{ |k| \geq n+1 \\ n \, \nmid \, k}} e^{-\pi \alpha k^2} k.
\end{align*}
We deduce that, again for $n$ sufficiently large,
$$\left\|  h \right\|_{L^{\infty}} \ll \frac{\| g_1\|_{L^2}}{n^{100}}.$$
We can now conclude the argument
\begin{align*}
    & \, \sum_{j=1}^{n} \theta_\alpha \left(x - \frac{j}{n} + \varepsilon_j \right) = A(x)+g_1(x) + g_2(x) + h(x)\\
    = & \, \sum_{j=1}^{n} \theta_\alpha \left(x - \frac{j}{n}  \right) + \mathcal{O}\left(n  e^{-\pi \alpha n^2}  \sum_{j=1}^{n} \varepsilon_j^2\right) + g_1(x) + g_2(x) + h(x).
\end{align*}
Applying all the prior results, for $n$ sufficiently large, we get
\begin{align*}
    \min_x f(x) & \leq \min_{0 \leq k \leq n-1} \sum_{j=1}^{n} \theta_\alpha \left( \frac{k+1/2}{n} - \frac{j}{n} + \varepsilon_j \right)\\
    & \leq   \left[\min_{0 \leq x \leq 1} \sum_{j=1}^{n} \theta_\alpha \left(x - \frac{j}{n}  \right) \right] +\mathcal{O}\left(n  e^{-\pi \alpha n^2}  \sum_{j=1}^{n} \varepsilon_j^2\right) \\
    & \qquad - \frac{\|g_1\|_{L^2}}{n^{3/2}} + \|g_2\|_{L^{\infty}} + \|h\|_{L^{\infty}} \\
    &\leq  \left[\min_{0 \leq x \leq 1} \sum_{j=1}^{n} \theta_\alpha \left(x - \frac{j}{n}  \right) \right]  - \frac{1}{2}\frac{\|g_1\|_{L^2}}{n^{3/2}}. 
\end{align*}
Recalling that
$$    \|g_1(x)\|_{L^2} \geq \frac{ e^{-\pi \alpha \left( \frac{n-1}{2} \right)^2}}{2} \left(\sum_{j=1}^{n} \varepsilon_j^2 \right)^{1/2}$$
we deduce that the minimal value of $f(x)$ is maximal if and only if
$$ \varepsilon_1 = \varepsilon_2 = \dots = \varepsilon_n = 0.$$
As the equidistributed points provide that the minimum is taken exactly in between them, we obtain equality in the last calculation and, hence, derive our main result.

\section{Small \texorpdfstring{$n$}{n} and shifting one point}\label{sec:n}
The case when $n$ is small needs, as mentioned in \S \ref{sec:intro}, new ideas. We have not tried to find solutions for say $n=3,4,5$ and we believe it is a hard problem. However, at least the case $n=2$, i.e., $\Gamma = \Z \cup (\Z + c)$, is fairly easy: the fact that $x=1/2$ gives the minimizer of $\theta_\alpha(x)$ suggests that we should place the second point exactly midway between the integers. It follows from Proposition \ref{pro:minimum_1/2} that we now have minima at $1/4$ and $3/4$ (in between the maxima at $0$, $1/2$ and $1$). Taking these as points of reference it is not hard to show that the equispaced distribution is optimal. In fact, this idea leads to the following generalization.
\begin{lem}
    Let $\alpha > 0$ be fixed, $x_1 \in [0,1)$ be arbitrary and $x_k = (k-1)/n$, for $k=2, \ldots n$. Then $p_\alpha$ has maximal polarization if and only if $x_1 = 0$. Equivalently, $\min_x f_\alpha(x)$ is maximal if and only if $x_1 = 0$, i.e., if the points are equispaced.
\end{lem}
\begin{proof}
    It is seen from the product formula \eqref{eq:triple_product} that $\theta_\alpha(x)$ is symmetric in $x$ and a decreasing function on $(0,1/2)$ (see also \cite{Mon88}). Hence, we have
    \begin{equation}
        \theta_\alpha(y) < \theta_\alpha(x_0) = \theta_\alpha(1-x_0), \qquad y \in (x_0,1-x_0), \, x_0 \in [0,1/2).
    \end{equation}
    Recall from Proposition \ref{pro:minimum_1/2} that
    $$
    \min_x \sum_{k=1}^{n} \theta_\alpha\left(x-\frac{k-1}{n}\right) = \sum_{k=1}^{n} \theta_\alpha\left( \frac{\ell + 1/2}{n} - \frac{k-1}{n} \right), \quad \ell = 0, \ldots , n-1.
    $$
    Now, we pick $\ell = 0$ and compare (taking periodicity into account) values at $1/(2n)$:
    \begin{align}
        \sum_{k=1}^{n} \theta_\alpha \left( \frac{1}{2n} - \frac{k-1}{n} \right)
        & = \theta_\alpha\left( \frac{1}{2n} \right) + \sum_{k=1}^{n-1} \theta_\alpha\left( \frac{1}{2n} - \frac{k}{n} \right) \\
        & > \theta_\alpha\left( \frac{1}{2n} + y \right) + \sum_{k=1}^{n-1} \theta_\alpha\left( \frac{1}{2n} - \frac{k}{n} \right), \quad y \in \left(0, 1-\frac{1}{n} \right).
    \end{align}
    The inequality holds true when shifting by $y \in (-1+1/n,0)$ (so periodically to the right) and picking $\ell = n-1$, by symmetry.
\end{proof}

\textbf{Acknowledgment.} We are sincerely grateful for a large number of very helpful suggestions made by two remarkably thoughtful referees.

\end{document}